\documentclass{article}
\usepackage{fullpage}

\usepackage{sseq}

\usepackage{amsmath}
\usepackage{amssymb}
\usepackage{stmaryrd}
\usepackage{epsfig}
\usepackage{psfrag}
\usepackage[matrix, arrow, curve]{xy}

\newcommand{\dlpullback}[1][dl]{\save*!/#1-1.5pc/#1:(-1,1)@^{|-}\restore}

\newcommand{\drpullback}[1][dr]{\save*!/#1-1.5pc/#1:(-1,1)@^{|-}\restore}

\usepackage{eucal}
\usepackage{calligra}

\usepackage{graphicx}
\usepackage{mathrsfs}
\usepackage{comment}

\usepackage{url}

\usepackage{footmisc}

\usepackage{amsthm}

\swapnumbers

\theoremstyle{definition}
\newtheorem{definition}{Definition}[section]
\newtheorem{example}[definition]{Example}
\newtheorem{remark}[definition]{Remark}

\theoremstyle{theorem}

\newtheorem{theorem}[definition]{Theorem}
\newtheorem{proposition}[definition]{Proposition}
\newtheorem{lemma}[definition]{Lemma}
\newtheorem{corollary}[definition]{Corollary}

\newtheoremstyle{numberfirst}
 {}
 {}
 {}
 {}
 {\bfseries}
 {}
 {.5em}
 {#2.~\thmnote{#3}}

\theoremstyle{numberfirst}

\newtheorem{taller}[definition]{$\!\!$}
\newenvironment{blanko}[1]%
{\begin{taller}{\bf #1}\normalfont}%
{\end{taller}}

\makeatletter
   \renewcommand{\section}{\@startsection {section}{1}{\z@}%
   {-3.5ex \@plus -1ex \@minus -.2ex}%
   {2.3ex \@plus.2ex}%
   {\normalfont\large\bfseries}}
\makeatother

\newcommand{\A}{\mathcal{A}}
\newcommand{\B}{\mathcal{B}}
\newcommand{\C}{\mathcal{C}}
\newcommand{\D}{\mathcal{D}}
\newcommand{\E}{\mathcal{E}}
\newcommand{\F}{\mathcal{F}}
\newcommand{\G}{\mathcal{G}}
\renewcommand{\H}{\mathcal{H}}
\newcommand{\K}{\mathcal{K}}

\newcommand{\M}{\mathcal{M}}
\newcommand{\Oooh}{\mathcal{O}}
\renewcommand{\P}{\mathcal{P}}
\newcommand{\Q}{\mathcal{Q}}
\renewcommand{\S}{\mathcal{S}}
\newcommand{\T}{\mathcal{T}}
\newcommand{\U}{\mathcal{U}}
\newcommand{\X}{\mathcal{X}}
\newcommand{\Y}{\mathcal{Y}}
\renewcommand{\AA}{\mathbb{A}}
\newcommand{\GG}{\mathbb{G}}
\newcommand{\PP}{\mathbb{P}}

\newcommand{\too}{\longrightarrow}
\newcommand{\op}{\mathrm{op}}
\DeclareMathOperator{\Aut}{Aut}
\DeclareMathOperator{\Cat}{Cat}
\DeclareMathOperator{\End}{End}
\DeclareMathOperator{\eq}{\, eq}
\DeclareMathOperator{\Fun}{Fun}
\DeclareMathOperator{\Gpd}{Gpd}
\DeclareMathOperator{\Ho}{Ho}
\DeclareMathOperator{\Hom}{Hom}
\DeclareMathOperator{\id}{id}
\DeclareMathOperator{\Map}{Map}
\DeclareMathOperator{\Eq}{Eq}
\DeclareMathOperator{\colim}{colim}
\DeclareMathOperator{\Pre}{Pre}
\DeclareMathOperator{\Set}{Set}
\DeclareMathOperator{\hocolim}{hocolim}

\newcommand{\isopil}{\stackrel{\raisebox{0.1ex}[0ex][0ex]{\(\sim\)}}%
			{\raisebox{-0.15ex}[0.28ex]{\(\rightarrow\)}}}

\newcommand{\loccart}{locally cartesian}
\newcommand{\Loccart}{Locally cartesian}
\newcommand{\qi}{$\infty$-quasi}
\newcommand{\Qi}{$\infty$-Quasi}

\newcommand{\simplC}{\D}

\newcommand{\sEq}{\mathcal{E}\!q}
\newcommand{\sMap}{\mathcal{M}\!ap}
\newcommand{\intEq}{\underline{\Eq}}
\newcommand{\intMap}{\underline{\Map}}
\newcommand{\earrow}{e}

\usepackage{color}

\begin{document}

\title{Univalence in locally cartesian closed $\infty$-categories}
\author{David Gepner and Joachim Kock}
\date{}
\maketitle

\begin{abstract}
  After developing the basic theory of locally cartesian
localizations of presentable locally cartesian closed $\infty$-categories, 
we establish the representability of equivalences and show that
univalent families, in the sense of Voevodsky, form a poset
isomorphic to the poset of bounded local classes, in the sense of
Lurie.  It follows that every $\infty$-topos has a hierarchy of ``universal'' univalent
  families, indexed by regular cardinals, and that $n$-topoi have univalent
  families classifying $(n-2)$-truncated maps.  We show that univalent families
  are preserved (and detected) by right adjoints to locally cartesian
  localizations, and use this to exhibit certain canonical univalent families in
  \qi topoi (certain $\infty$-categories of ``separated presheaves'',
  introduced here).  We also
  exhibit some more exotic examples of univalent families, illustrating that a
  univalent family in an $n$-topos need not be $(n-2)$-truncated, as well as
  some univalent families in the Morel--Voevodsky $\infty$-category of motivic
  spaces, an instance of a locally cartesian closed $\infty$-category which is not
  an $n$-topos for any $0\leq n\leq\infty$.  Lastly, we show that any
  presentable locally cartesian closed $\infty$-category is modeled by a
  combinatorial type-theoretic model category, and conversely that the $\infty$-category
  underlying a combinatorial type-theoretic model category is presentable and locally
  cartesian closed.  Under this correspondence, univalent families in
  presentable locally cartesian closed $\infty$-categories correspond to univalent
  fibrations in combinatorial type-theoretic model categories.
\end{abstract}

\tableofcontents

\section*{Introduction}

The connection between type theory and homotopy theory goes back
to Hofmann and Streicher~\cite{Hofmann-Streicher:98}, with
subsequent advances made by Awodey and
Warren~\cite{Awodey-Warren:0709.0248}, Gambino and
Garner~\cite{Gambino-Garner:0803.4349}, van den Berg and
Garner~\cite{vdBerg-Garner:0812.0298},
Lumsdaine~\cite{Lumsdaine:0812.0409}, and others.  It has reached a new
level of significance with Voevodsky's Univalence Axiom
\cite{VoevodskyV:notts}, which roughly stipulates that
intensional identity is homotopy equivalence.  This 
comes with the prospect of providing
a new foundation for mathematics in which the notion of
homotopy is built in at the foundational level, reflecting the important
mathematical practice of identifying algebraic structures if they are
isomorphic, categories if they are equivalent, and so on \cite{HoTT-book}.  On
the categorical level, the existence of a univalent universe also solves the
so-called coherence problem in categorical semantics of type theory: a coherent
choice of all type-theoretic operations can be made in terms of the universe and
exploiting its universal property.  Voevodsky explained his Univalent
Foundations Program in a series of 9 lectures in Oberwolfach in 2011
\cite{oberwolfach2011}, and established in particular a categorical model
containing a univalent universe: it is a certain Kan fibration in the category
of simplicial sets, and its univalence property is established using Quillen
model structures and a well-ordering trick.  A related construction is due to
Streicher~\cite{Streicher:MR3166196}, and some simplifications in
Voevodsky's proof were provided by Joyal and by Moerdijk.  A concise and
self-contained exposition of the proof was recently given by Kapulkin, Lumsdaine
and Voevodsky~\cite{Kapulkin-Lumsdaine-Voevodsky:1203.2553}.  
Shulman~\cite{Shulman:1203.3253}, building on Voevodsky's result, showed that
categories of inverse diagrams of simplicial sets provide other models, and
recently Shulman~\cite{Shulman:1307.6248} and Cisinski~\cite{Cisinski:1406.0058}
have found further models based on elegant Reedy categories.
The difficulty encountered in constructing these univalent universes
stems from the fact that many features of type theory are not invariant under
homotopy equivalence, which makes constructions delicate.  But in fact, the
univalence axiom itself is a very robust notion (and homotopy invariant, in particular).  

In the present contribution, as a complement to the original
type-theoretic, syntactic account of univalence 
exposed in detail in \cite{HoTT-book}, and
abstracting from the simplicial developments in
\cite{Kapulkin-Lumsdaine:1211.2851},
\cite{Kapulkin-Lumsdaine-Voevodsky:1203.2553}, we take the more intrinsic
viewpoint of presentable locally cartesian closed $\infty$-categories to study
univalence in its own right.
We relate it with concepts and results
of more geometric origin, with the aim of placing the notion in a broader
perspective.
While the most obvious examples of univalent families arise in $\infty$-topoi, the
most natural setting for the notion of univalence seems to be that of
(presentable) locally cartesian closed $\infty$-categories: not only because of the
distinguished role of locally cartesian closed categories in type theory, but
more so because the abstract notion of univalence is meaningful and behaves well
precisely in this context.  We prove some general results about univalent
families, related to local classes (in the sense of Lurie), factorization
systems, truncation, localizations and \qi topoi, and provide significant
examples of univalent families outside the realm of $\infty$-topoi.  Along the way
we develop general theory of presentable locally cartesian closed
$\infty$-categories, expanding upon the material found in Lurie's {\em Higher Topos
Theory} \cite{Lurie:HTT}, and generalizing basic results from the $\infty$-topos
case.

We emphasize that the majority of the methods and results of this paper are at the level of $\infty$-categories.
Nevertheless, in the final section we do take a few steps towards
obtaining strict models of type theory.
We show that any univalent family in a presentable locally cartesian closed
$\infty$-category can be lifted,
by a fairly
standard procedure, to the model categories where the current semantics of type
theory takes place.  We refer to such model categories as ``combinatorial type-theoretic'' 
model categories (cf.~Definition \ref{def:typemodel}). It is important 
to keep in mind that (even in the case of $\infty$-topoi) the univalent 
fibrations produced are typically not universes in the usual strict sense,
and hence 
do not provide new models validating the univalence axiom in type theory
in the same explicit way
that Voevodsky's, Shulman's, and Cisinski's constructions do.  
It is expected that every $\infty$-topos should be a model for
type theory in some sense,
and we hope the
present $\infty$-categorical foundations will prove useful to
establish this.  More generally, we speculate that eventually a genuinely
$\infty$-categorical semantics for type theory will be developed, bypassing
altogether the subtleties inherent in model categorical semantics.

\medskip

\noindent
{\bf Main results.}
We begin by characterizing presentable locally
cartesian closed $\infty$-categories as precisely the accessible locally cartesian
localizations\footnote{Called 
{\em semi-left-exact localizations} by 
Cassidy--H\'ebert--Kelly~\cite{Cassidy-Hebert-Kelly} for 
$1$-categories.} of $\infty$-topoi (cf.~Corollary~\ref{lccc=lccc}).  This is exploited 
throughout.
In Section~\ref{sec:Eq} we consider certain naturally
occurring presheaves of spaces defined on various slices $\C_{/S}$ of a locally
cartesian closed $\infty$-category $\C$: namely, the mapping sheaf $\sMap_{/S}(X,Y)$
(whose value at $T\in\C_{/S}$ is $\Map_{/T}(X\times_S T,Y\times_S T)$) and its
subfunctor $\sEq_{/S}(X,Y)$ of equivalences. The main result here is this:

\medskip

\noindent
{\bf Theorem.} (cf.~Theorem \ref{thm:eqrep})
{\em
  Let $S$ be an object in a locally cartesian closed
  $\infty$-category $\C$.  Then for each pair of objects $X \to S$ and $Y \to S$ of
  $\C_{/S}$, the right fibration $\sEq_{/S}(X,Y) \to \C_{/S}$ is a sheaf. If in addition $\C$ is presentable then the right fibration $\sEq_{/S}(X,Y)\to\C_{/S}$ is representable.}
\medskip

The representing object $\intEq_{/S}(X,Y)\to S$ is called the {\em internal
object of equivalences}, a subobject of the internal mapping object
$\intMap_{/S}(X,Y)\to S$.  This representability is important in its own
right for several reasons.  In the present context, it is simply a necessary 
prerequisite for being able to define univalence in a more
general setting.  A map $p:X\to S$ in $\C$ (which we think of as an $S$-indexed
family of objects of $\C$) determines, by pullback along the projections
$S\times S\to S$, a pair of objects $p_1^*X$ and $p_2^*X$ of $\C_{/S\times S}$,
and we write $\intEq_{/S}(X)=\intEq_{S\times S}(p_1^*X,p_2^*X)$ for the
resulting equivalence object.  This objects sits under the diagonal
$\delta_S\in\C_{/S\times S}$, and we say that the family $p:X\to S$ is {\em
univalent} if $\delta_S\to\intEq_{/S}(X)$ is an equivalence.

\medskip

\noindent {\bf Theorem.}~(cf.~Theorem \ref{poset}) 
{\em 
  For any presentable locally cartesian closed $\infty$-category $\C$, the
  equivalence classes of univalent families in $\C$ form a (possibly large)
  poset which is canonically isomorphic to the poset of bounded local
  classes of maps in $\C$.
}

\medskip

\noindent
A bounded local class of maps (see \cite[Section 6.1]{Lurie:HTT}) is roughly a class of maps which is closed under basechange and satisfies a certain sheaf condition.
Lurie shows\footnote{\cite{Lurie:HTT}, theorems 6.1.0.6 and
6.1.3.9} that in an $\infty$-topos $\X$ the class of all maps is local, hence it
follows that for each regular cardinal $\kappa$ the classifying family for
relatively $\kappa$-compact maps in $\X$ is a univalent family, an ``internal
universe''.  The close analogy between univalence and the $\infty$-topos axioms was
perhaps first pointed out by Joyal.  The univalence condition can be interpreted
as a descent property which allows one to glue together families into ``moduli
spaces'', and having this descent property for {\em all} families characterizes
precisely the $\infty$-topoi amongst the presentable locally cartesian closed
$\infty$-categories.

Theorem~\ref{poset} implies that in an $\infty$-topos, every univalent family 
is a subfamily of a universal family, i.e. a classifying family for relatively $\kappa$-compact maps.  This characterization extends to
general presentable locally cartesian closed $\infty$-categories by virtue of the following
result.

\medskip

\noindent
{\bf Theorem.}~(cf.~Theorem \ref{thm:induce-along-R})
{\em
  Let $L: \P\to\C \subset \P$ be an accessible locally cartesian
  localization of a presentable locally cartesian closed $\infty$-category
  $\P$.  Then $p:X\to S$ is a univalent family in $\C$ if
   and only if it is also a univalent family in $\P$.
}

\medskip

\noindent

Since any presentable locally cartesian closed $\infty$-category
embeds in this way into an $\infty$-topos, this gives some control over univalent
families in general, and shows in particular
that every univalent family must be
a subfamily of a universal family of the ambient $\infty$-topos.

Factorization systems play an important role in our treatment.  In
Section~\ref{sec:fact} we first establish the ($n$-connected, $n$-truncated)
factorization system in any presentable $\infty$-category $\C$ 
(Proposition~\ref{conntrunc}), and show that it is
stable when $\C$ is locally cartesian closed 
(Proposition~\ref{nconnstable}).
We obtain a rich supply of local classes in
an $\infty$-topos
by showing that  the right orthogonal class of any stable factorization system is local.
This result leads
to a simple proof of the fact that
in an $n$-topos, the class $\T_k$ of $k$-truncated maps is local for all 
$-2 \leq k \leq n-2$ (cf. Corollary~\ref{k<=n-2}).

While the locally cartesian closed $\infty$-categories correspond to locally
cartesian  localizations of $\infty$-topoi, there are a
number of other stronger exactness conditions of interest in practice, yielding
classes of $\infty$-categories intermediate between $\infty$-topoi and presentable
locally cartesian closed $\infty$-categories.  In Section~\ref{sec:quasi} we introduce
a general notion of \qi topos, by which we mean $\infty$-categories of
$\F$-separated objects for a suitable factorization system $(\E,\F)$, and with
respect to a compatible left-exact
localization of an ambient $\infty$-topos.  
This notion specializes to the classical notion of quasitopos in the
analogous $1$-categorical situation when $(\E,\F)$ is the epi-mono factorization 
system.

\medskip

\noindent
{\bf Theorem.} (cf.~Theorem \ref{thm:PQE})
{\em
  Let $\X$ be an $\infty$-topos, presented as a left-exact localization 
  $L:\P\to\X\overset{G}\subset\P$ of an $\infty$-topos $\P$ equipped with a
  stable factorization system 
  $(\E,\F)$ such that $GL$ preserves $\F$,
  and let $\Q\subset\P$ denote the full subcategory of $\F$-separated objects.
  Then the inclusion functors $\overline{G}:\X\to\Q$ and $G':\Q\to\P$ admit
  left adjoints $\overline{L}:\Q\to\X$ and $L':\P\to\Q$, respectively, such
  that $\overline{L}$ is left-exact and $G'L'$ preserves the factorization
  system $(\E,\F)$ and pullbacks over $L'$-local objects (so in particular 
  $L'$ is locally cartesian).
}

\medskip

\noindent
In this situation
it now follows from Theorem~\ref{thm:induce-along-R} that the 
canonical univalent families in the small $\infty$-topos induce
univalent families in the \qi topos, somewhat analogous to the
strong-subobject classifiers in classical quasitopoi.

While the constructions so far are ``top-down'', it is also interesting to
construct univalent families ``bottom-up'', which is the topic of
Section~\ref{sec:bundles}.  The smallest univalent families are bundle
classifiers: for every object $F$ of an $\infty$-topos, the universal bundle with
fiber $F$ is univalent (see~\ref{Fbundles}).  Bigger families can be obtained by
taking unions.  This viewpoint leads to some unexpected univalent families,
exemplifying in particular that a univalent family in an $n$-topos need not be
$(n-2)$-truncated (see~\ref{BG}).  We also exhibit some univalent families in
the locally cartesian closed $\infty$-category of motivic spaces
\cite{Morel-Voevodsky:IHES} constructed from certain group schemes (see \ref{Gm}
and \ref{E}), providing further examples of univalent families outside of topoi.

The current categorical semantics of type theory involves certain strict
fibration properties, which are necessary to get a literal interpretation of the
syntactic rules.  These strictness features have no intrinsic homotopical
content, but can be formulated in terms of $1$-categorical notions in a Quillen
model category.  We establish in the final Section~\ref{sec:type} that one can
always model univalent families in presentable locally cartesian closed
$\infty$-categories by univalent fibrations in what we call combinatorial type-theoretic model
categories, namely right-proper combinatorial Quillen model categories whose
underlying category is locally cartesian closed, and whose cofibrations are
precisely the monomorphisms.  The main result is this:

\medskip

\noindent
{\bf Theorem.} (cf.~Theorem \ref{lcc=>type})
{\em
  Every presentable locally cartesian closed $\infty$-category $\C$ can be
  presented as a combinatorial type-theoretic model category $\M$ (that is, the underlying
  $\infty$-category of $\M$ is equivalent to $\C$).
}

\medskip

{Furthermore, univalent families in the $\infty$-setting correspond
to univalent fibrations at the model level (Proposition~\ref{prop:univ=univ}).}
In the case of an $\infty$-topos, the object classifier {thus} lifts to a
univalent fibration in the model category, but it should be noted that it is not
a universe in the usual strict sense of type theory \cite{Shulman:1203.3253}.
One issue is that 
the type-forming operations on the model category are generally not
strictly stable under pullback, which seems necessary to get a model of
type theory.
This issue can be
fixed by means of the recent Lumsdaine--Warren coherence
theorem~\cite{Lumsdaine-Warren:1411.1736}.  The second issue is that the
universe obtained is only a {\em weak Tarski universe} in the sense of Shulman
(see \cite{nLab:}).  Roughly this means that the elements of the universe are
not themselves types, but are rather codenames for types, and that the universe
is equipped with operations mimicking the type formers, which under the codename
function are compatible with the real type-former operations up to 
equivalence (i.e.~homotopy).  It is an important question in type theory to figure
out to what extent weak Tarski universes can be used instead of strict ones, and
in particular how the weakening interferes with computational properties.

\bigskip

\noindent
{\bf Acknowledgements.}
This work was originally prompted by the preprints of
Kapulkin, Lumsdaine and Voevodsky~\cite{Kapulkin-Lumsdaine-Voevodsky:1203.2553}
and Shulman~\cite{Shulman:1203.3253}.  We have also benefited from 
conversations with
Peter Arndt, Steve Awodey, Nicola Gambino, 
Andr\'e Joyal, Raffael Stenzel, Urs Schreiber, Markus Spitzweck, and especially Mike 
Shulman.  {We thank the anonymous referees for pertinent remarks that
helped improve the exposition.}
The first author was partially supported by NSF grant DMS-1406529.
The second author was partially supported by grants
MTM2009-10359, 
MTM2010-20692 and MTM2103-42293-P 
              of Spain and by
SGR1092-2009 
                 of Catalonia.

\bigskip

\noindent
{\bf Terminology and notation.}
For simplicity we adhere to the terminology and
notation of Lurie [HTT] as much as possible.
It should be noted, though, that while Lurie often works in 
simplicial sets with the Joyal model structure, we work consistently
in the $\infty$-category of $\infty$-categories, as our statements and proofs
are invariant under equivalence, and indeed most objects are given
in terms of universal properties and therefore are determined only up to 
equivalence anyway.  Hence for example we deal with slice categories
and mapping spaces as $\infty$-categories, rather than distinguishing 
between the various non-isomorphic but equivalent 
simplicial sets that can be chosen to represent
these notions in the Joyal model structure.

\section{Locally cartesian closed $\infty$-categories}
\label{sec:lcc}

\begin{blanko}{Presentable locally cartesian closed $\infty$-categories.}
  Recall that an $\infty$-category $\C$ is {\em locally cartesian closed} if 
  $\C$ has pullbacks and, for any morphism $f:T\to S$ in $\C$, the associated 
  pullback functor
  $f^*:\C_{/S}\to\C_{/T}$ admits a right adjoint, {which is then 
  denoted $f_*$; the left adjoint is denoted $f_!$}.
  We will often assume that $\C$ is presentable; in this case, by the adjoint functor
  theorem,\footnote{[HTT 5.5.2.9]}
  since slices of presentable $\infty$-categories are again presentable,
  being locally cartesian closed is equivalent to colimits being
  universal, the condition used in Lurie~[HTT].
\end{blanko}

\begin{blanko}{\Loccart\ localizations.}
  Just as left-exact localizations are a central notion in topos theory,
the notion of locally cartesian localization,\footnote{Called 
{\em semi-left-exact localization} by Cassidy--H\'ebert--Kelly~\cite{Cassidy-Hebert-Kelly} for 
$1$-categories.} which we now introduce, plays
a similar role for locally cartesian closed $\infty$-categories.
  Let $\P$ be a presentable locally cartesian closed $\infty$-category and let $L :
  \P\to\C \subset \P$ be an accessible localization, with right adjoint
  inclusion functor $G:\C\to\P$.  We refer to the objects of $\C$ as 
  local objects.
  For each local object $S$ there is induced a localization functor
  $L_S: \P_{/S} \to \C_{/S}\subset \P_{/S}$ (with right adjoint inclusion
  functor $G_S$).
  For any map $f:T\to S$ between local objects we have
  commutative diagrams
$$
\xymatrix{
\P_{/S} \ar[r]^{L_S}  &\C_{/S}  \\
\P_{/T} \ar[u]^{f_!}\ar[r]_{L_T} & \C_{/T}\ar[u]_{f_!}
}
\qquad
\xymatrix{
\P_{/S}  \ar[d]_{f^*} &\ar[l]_{G_S}\C_{/S} \ar[d]^{f^*} \\
\P_{/T} & \C_{/T}\ar[l]^{G_T}  .
}
$$
We say that the localization $L$ is {\em \loccart} if it
commutes with basechange between local objects.
In other words, for all $f:T\to S$ in $\C$, the diagram
$$\xymatrix{
\P_{/S} \ar[r]^{L_S} \ar[d]_{f^*} &\C_{/S} \ar[d]^{f^*} \\
\P_{/T} \ar[r]_{L_T} & \C_{/T}
}$$
commutes.
Equivalently, for every diagram  $T \to S \leftarrow X$ in $\P$ such that $S$ and $T$
are local objects, the natural map
$$L(T\times_S X)\too LT\times_{LS} LX \simeq T\times_S LX$$ is an equivalence.
A morphism $\alpha$ in $\P$ is called a local equivalence when
$L(\alpha)$ is an equivalence in $\C$.  Similarly, for $S$ a local object,
a morphism $\alpha$ in 
$\P_{/S}$ is called a local equivalence when $L_S(\alpha)$ is an equivalence in 
$\C_{/S}$.
\end{blanko}

\begin{proposition}\label{loccartstars}
Let $\P$ be a presentable locally cartesian closed $\infty$-category and 
let $L:\P\to\C\subset\P$ be an accessible localization.  The following conditions are 
equivalent:
\begin{enumerate}
  \item $L$ is \loccart.

  \item For every morphism $f:T \to S$ between local objects, the functor
  $f^*:\P_{/S} \to \P_{/T}$ preserves local equivalences.

  \item For every morphism $f:T \to S$ between local objects, the functor 
  $f_*:\P_{/T}\too\P_{/S}$
sends $L_T$-local objects to $L_S$-local objects.
\end{enumerate}
\end{proposition}

\begin{proof}
  (1)$\Rightarrow$(2): Suppose $L$ preserves basechange between local objects.
  If $\alpha$ is a local equivalence in $\P_{/S}$ then $L_S(\alpha)$ and hence 
  $f^*(L_S(\alpha))$ are 
  equivalences.  But the latter is equivalent to $L_T( f^*(\alpha) )$ by
  assumption, which is to say that $f^*(\alpha)$ 
  is a local equivalence.
  Hence $f^*$ preserves local equivalences.
  
  (2)$\Rightarrow$(3): Assuming that $f^*$ preserves local equivalences, let 
  $z:Z\to T$ be a local object in $\P_{/T}$.
  We must check that $f_*(z)$ is again a local object, now in $\P_{/S}$.
  For any local equivalence
  $\alpha:p \to 
  p'$ in $\P_{/S}$,
  we have a diagram
  $$
  \xymatrix{
  \Map_{/T}(f^*(p'), z) \ar[d]_\simeq\ar[rr] && \Map_{/T}(f^*(p),z)
  \ar[d]_\simeq\\
  \Map_{/S}(p', f_*(z)) \ar[rr] && \Map_{/S}(p,f_*(z)) ,
  }$$
  where the vertical maps are equivalences by adjunction.  Since $f^*$ preserves 
  local equivalences, the top horizontal map is an equivalence, and hence the bottom 
  horizontal map is an equivalence.  Since this is true for all local 
  equivalences $\alpha:p\to p'$, this is to say that $f_*(z)$ is a local object.
  Hence $f_*$ preserves local objects.

  (3)$\Rightarrow$(1):
  To say that each $f_*$ preserves local objects means that
  it is compatible with the inclusion of the slice in $\C$ to the slice in $\P$, so as to give a commutative 
  square of right adjoints
  $$
  \xymatrix{
\P_{/S}   &\ar[l]_{G_S}\C_{/S}  \\
\P_{/T} \ar[u]^{f_*}& \C_{/T} .\ar[l]^{G_T} \ar[u]_{f_*}
}
$$
The analogous square of left adjoints then also commutes, which is 
precisely what it means for $L$ to be \loccart.
\end{proof}

\begin{proposition}
  Let $\P$ be a presentable locally cartesian closed $\infty$-category, and let
  $L:\P\to\C\subset \P$ be an accessible \loccart\ localization.  Then 
  $\C$ is a presentable locally cartesian closed $\infty$-category.
\end{proposition}

\begin{proof}
  This follows by inspection of the proof of \cite[Lemma 6.1.3.15]{Lurie:HTT}:
  there it is proved that a left-exact localization preserves the property of
  being locally cartesian closed, but in fact the only place left-exactness is
  used is to preserve a certain pullback square where in fact one of the legs is
  between local objects.
\end{proof}

Recall that, for any presentable $\infty$-category $\C$ and any 
sufficiently large regular cardinal $\kappa$, the full subcategory 
$\C^\kappa\subset\C$ of $\kappa$-compact objects
induces a colimit-preserving functor
\[
\Pre(\C^\kappa)\too\C ,
\]
which is essentially surjective and admits a fully faithful right 
adjoint.\footnote{\cite[Theorem 5.5.1.1]{Lurie:HTT}}
We refer to this as the ``standard presentation'' (even though it 
involves a choice of cardinal).

\begin{proposition}\label{lcc=lcc}
If $\C$ is a presentable locally cartesian closed $\infty$-category then
its standard presentation $L: \Pre(\C^\kappa) \to \C$ is an accessible
\loccart\ localization.
\end{proposition}

\begin{proof}
  The statement asserts that the localization is locally cartesian 
  whenever $\kappa$ is large enough for the localization functor
  $L:\Pre(\C^\kappa)\too\C\subset\Pre(\C^\kappa)$ to exist.
  We shall apply criterion (2) of Proposition~\ref{loccartstars}.
  Let $F:T\to S$ be a map of local objects and
  let $\kappa'\geq\kappa$ be a cardinal such that both $S$ and $T$ are in $\C^{\kappa'}$.
  By the commutativity of the diagram
  \[
  \xymatrix{
  \C_{/T}\ar[r] & \Pre(\C^\kappa)_{/T}\ar[r]\ar[d]^{F_*} & \Pre(\C^{\kappa'})_{/T}\ar[d]^{F_*}\\
  \C_{/S}\ar[r] & \Pre(\C^\kappa)_{/S}\ar[r] & \Pre(\C^{\kappa'})_{/S},}
  \]
  and since the horizontal functors are fully faithful,
  we may suppose without loss of generality that $\kappa'=\kappa$, so that
  $F:T \to S$ is a map of representable presheaves and therefore is of the form 
  $F\simeq\Map_{\C^\kappa}(-,f)$ for some $f:s\to t$ in $\C^\kappa$.
  We need to show that $F^*$ applied to 
  any generating local equivalence
  \[
  \colim_i\Map_{\C^\kappa_{/S}}(-,x_i)\too\Map_{\C^\kappa_{/S}}(-,\colim_i x_i), 
  \]
  is again a local equivalence.
  But this is clear: in the diagram
  \[
  \xymatrix{
  F^*\colim_i \Map_{\C^\kappa_{/S}}(-,x_i)  \ar[d] \ar[r] &
    \colim_i F^*\Map_{\C^\kappa_{/S}}(-,x_i) \ar[r]^-\sim &
    \colim_i\Map_{\C^\kappa_{/T}}(-,f^*x_i) \ar[d] \\
  F^*\Map_{\C^\kappa_{/S}}(-,\colim_i x_i) \ar[r]^-\sim &
    \Map_{\C^\kappa_{/T}}(-,f^*\colim_i x_i) \ar[r] & 
    \Map_{\C^\kappa_{/T}}(-,\colim_i f^*x_i)
  }
  \]
  the left vertical map is the map in question, and the right vertical map is
  itself a generating local equivalence.
  The top left horizontal map is an equivalence because $F^*$ is a left adjoint,
  since a presheaf category is locally cartesian closed.
  The bottom right horizontal map is an equivalence because $\C$ and hence $\C^\kappa$
  is locally cartesian closed, 
  hence $f^*:\C^\kappa_{/S}\to\C^\kappa_{/T}$ preserves colimits.
  The two maps already displayed as equivalences are so because $F$ is 
  represented by $f$.
  Finally, $L$ is accessible as it is a left adjoint functor between accessible 
  $\infty$-categories.\footnote{\cite[Proposition 5.4.7.7]{Lurie:HTT}}
\end{proof}

\begin{corollary}\label{lccc=lccc}
  An $\infty$-category $\C$ is presentable and locally cartesian closed if and only
  if there exists a small $\infty$-category $\D$ such that $\C$ is an accessible
  locally cartesian localization of $\Pre(\D)$.
  \qed
\end{corollary}

Recall (from [HTT 5.2.5.1]) that if
$L:\X\to\C$ is left adjoint to $G:\C\to\X$ with 
unit $\eta: \id \Rightarrow GL$ and counit $\varepsilon:LG \Rightarrow \id$,
and if $\X$ has pullbacks, then for
each object $T\in \X$ the induced functor $L_T: \X_{/T} \to \C_{/LT}$ has again a right
adjoint $(\eta_T)^*\circ G_{LT} : \C_{/LT}\to \X_{/GLT}\to \X_{/T}$,
with counit  (at the object $Z \to LT$)  given by
\begin{equation}\label{counit}
\xymatrix{
L(T \times_{GLT} GZ) \ar[r]^-{L(\mathrm{pr}_2)} & LGZ \ar[r]^-{\varepsilon_Z}&Z .
}  
\end{equation}

\begin{lemma}\label{lem:sliceadjoint}
  Let $\X$ be a locally cartesian closed $\infty$-category, and
  let $L:\X\to\C \subset \X$ be a \loccart\ localization, with right adjoint 
  inclusion functor $G:\C\to\X$ and unit $\eta: \id \Rightarrow GL$.
  Then for each object $T\in\X$, the functor
  $(\eta_T)^*\circ G_{LT} : \C_{/LT}\to \X_{/GLT}\to \X_{/T}$ (right 
  adjoint to $L_T:\X_{/T}\to\C_{/LT}$)
  is fully faithful.
\end{lemma}

\begin{proof}
  We need to check that the counit component of diagram \eqref{counit} above is an equivalence.
  But since $L$ is \loccart, the first map is just the projection 
  $L(T \times_{GLT} GZ) \simeq LT \times_{LT} LGZ \to LGZ$, which is clearly an 
  equivalence, and the second map, $\varepsilon_Z$, is an equivalence since $G$ 
  is fully faithful.
\end{proof}

Recall from [HTT 5.5.6.18] that if $\C$ is a presentable 
$\infty$-category then the subcategory $\tau_{\leq k}\C \to \C$ of 
$k$-truncated objects has an accessible left adjoint, i.e.~an accessible 
localization functor $L : \C \to \tau_{\leq k}\C$.
\begin{lemma}\label{taulcc}
  For $\C$ a presentable locally cartesian closed $\infty$-category,
  the truncation functor $L : \C \to \tau_{\leq k}\C$ is 
  (accessible and) \loccart.
\end{lemma}
\begin{proof}
  To establish that $L$ is \loccart,
  we need to check, for any map $f:T\to S$ between $k$-truncated objects,
  that the diagram
  $$\xymatrix{
\C_{/S} \ar[r]^-{L_S} \ar[d]_{f^*} &(\tau_{\leq k}\C)_{/S} \ar[d]^{f^*} \\
\C_{/T} \ar[r]_-{L_T} & (\tau_{\leq k}\C)_{/T}
}$$
commutes.  But we have natural equivalences $(\tau_{\leq k}\C)_{/S} \simeq 
\tau_{\leq k}(\C_{/S})$ (e.g.~by [HTT 5.5.6.14]),
and the equivalent diagram
  $$\xymatrix{
\C_{/S} \ar[r]^-{\tau_{\leq k}} \ar[d]_{f^*} &\tau_{\leq k}(\C_{/S}) 
\ar[d]^{f^*} \\
\C_{/T} \ar[r]_-{\tau_{\leq k}} & \tau_{\leq k}(\C_{/T})  ,
}$$
commutes by [HTT 5.5.6.28] since $f^*$ preserves colimits and 
finite limits, $\C$ being locally cartesian closed.
\end{proof}
  
\begin{blanko}{Monomorphisms.}\label{lemma:mono123}
  Recall that a map $f: X \to Y$ in an $\infty$-category is said to be a {\em
  monomorphism} if it is $(-1)$-truncated (i.e.~for any object $T$,
  the induced map of $\infty$-groupoids $f_!:\Map(T,X) \to \Map(T,Y)$ is $(-1)$-truncated).
  This is equivalent to the condition that the diagonal
  $\delta_f: X \to X\times_Y X$ is an equivalence.  It is also easy to see that
  $f:X \to Y$ is mono if and only if the diagram
    $$\xymatrix{
    X \ar[r]^f \ar[d]_{\delta_X} & Y \ar[d]^{\delta_Y} \\
    X \times X \ar[r]_{f\times f} & Y \times Y
    }$$
    is a pullback.
\end{blanko}

\section{Objects of equivalences}
\label{sec:Eq}

{Given any pair of objects $X$ and $Y$ of an $\infty$-category $\C$,
there is the full subspace $\Eq(X,Y)_{\C}\subset\Map_{\C}(X,Y)$ consisting of the equivalences.  We are concerned in this section with a relative version 
of this, namely a right fibration $\sMap(X,Y) \to \C$ and a subfibration
$\sEq(X,Y) \subset \sMap(X,Y)\to\C$, 
which under certain assumptions will be shown to be representable by objects of
$\C$ denoted $\intEq(X,Y) \subset \intMap(X,Y)$.  The difference in typography 
for these spaces, fibrations and internal objects should be noted in the 
following.}

Given any pair of objects $X$ and $Y$ of an $\infty$-category $\C$ with 
{finite limits,}
there is a canonical right fibration
$$
\sMap(X,Y) \too \C 
$$
whose fiber over $T$ is $\Map_{/T}(X\times T, Y\times T)\simeq \Map(X\times 
T,Y)$.
It is defined as the restriction
$$\xymatrix{
\sMap(X,Y) \drpullback \ar[r] \ar[d] 
\ar[d] & \C_{/Y} \ar[d] \\
\C \ar[r]_{X \times -} & 
\C 
}$$
of the projection $\C_{/Y}\to\C$ (a right fibration since $\C$ has pullbacks)
along the ``multiplication by $X$'' map $\C\to\C$.  If this right fibration is
representable then the representing object, denoted $\intMap(X,Y)$, is the
internal mapping object; that is,
$\Map(T,{\intMap(X,Y)})\simeq\Map(X\times T, Y)$.

\begin{proposition}
  An $\infty$-category $\C$ is cartesian closed if and only if, for every pair of objects $X$ and $Y$
  of $\C$, the right fibration $\sMap(X,Y)\to\C$ is representable.
  \end{proposition}

\begin{proof}
  To say that $\C$ is cartesian closed is precisely to say that, for each object $X$, the
  functor $X \times -: \C\to\C$ is a left adjoint.
  Since pullback along left adjoints preserves representable right fibrations and $\C_{/Y}\to\C$ is a representable right fibration, it follows that $\sMap(X,Y)\to\C$ is a representable right fibration.
  
  Conversely, to show that the functor $X \times -:\C\to\C$ is a left adjoint, we must show
  that the associated cocartesian fibration (obtained by unstraightening the functor $\Delta^1\to\Cat_\infty$ corresponding to $X\times -:\C\to\C$)
  $$
  \xymatrix{
  \C\drpullback\ar[r]\ar[d] & \M \ar[d]^q & \C\dlpullback\ar[l]\ar[d]\\
  \{0\}\ar[r] & \Delta^1 & \{1\}\ar[l]}
  $$
  is also a cartesian fibration.  For this we 
  need to
  provide a cartesian lift of the unique non-identity arrow $\earrow:0\to 1$ 
  in $\Delta^1$ to an arrow of $\M$ with target an arbitrary object $Y$ over $1$.  The source of this arrow is
  going to be $\intMap(X,Y)$, and to provide the arrow itself we can exploit
  the fact that we already know that $q$ is a cocartesian fibration: it is hence enough
  to provide an arrow $X \times \intMap(X,Y)  \to Y$.  For this we take
  the map corresponding to the identity map of $T= \intMap(X,Y)$ under
  the equivalence 
  $$
  \Map( X\times T ,Y) \simeq \Map(T, \intMap(X,Y))
  $$
  expressing that $\intMap(X,Y)$ represents the right fibration.
  Finally we need to show that this arrow is cartesian.  By Lemma~\ref{cart} below,
  we have to show that evaluation
  $$
  \xymatrix{\Map(X,Y) \ar[r]^-{\text{cocart}} & X \times \intMap(X,Y)\ar[r]^-{\text{ev}} & Y}
  $$
  is a terminal object in the pullback
  $$\xymatrix{
   (\M\times_{\Delta^1}\{0\})_{/Y}\ar[r] \ar[d] \drpullback & \M\times_{\Delta^1}\{0\} \ar[d] \\
   \M_{/Y} \ar[r] & \M .
   }$$
  Choose $T\in (\M\times_{\Delta^1}\{0\})_{/Y}$; note that it has to belong to 
  the fiber over $0$.  Then we have a pullback square
  $$\xymatrix{
  \Map(T, \intMap(X,Y)) \drpullback \ar[r] \ar[d] & 
  \Map_{\C}(T, \intMap(X,Y))  \ar[d] \\
  \Map_{/Y}(T, \intMap(X,Y)) \ar[r] & \Map_{\M}(T, 
  \intMap(X,Y))
  }$$
  in which clearly the right vertical arrow is an equivalence.  The bottom left 
  space is calculated as the fiber of
  $$
  \Map_{\M}(T, \intMap(X,Y)) \to \Map_{\M}(T,Y),$$
  which is equivalent to the identity of
  $\Map(X\times T, Y)$, because $\Map_{\M}(T,Y)$ is equivalent to 
  $\Map(X\times T, Y)$ in the fiber over $1$.
\end{proof}

\begin{lemma}\label{cart}
  An inner fibration $p: \M \to \Delta^1$ is a cartesian fibration if and
  only if for every $t\in \M\times_{\Delta^1}\{1\}$ the $\infty$-category 
  $\M_{/t}\times_{\Delta^1}\{0\}$ (of lifts of the 
  nondegenerate edge $\earrow$ in $\Delta^1$ to $\M$ with target $t$), defined as the
  pullback
  $$\xymatrix{
  \M_{/t}\times_{\Delta^1}\{0\}\drpullback \ar[r] \ar[d] & \M\times_{\Delta^1}\{0\}\ar[d]\\
  \M_{/t} \ar[r] & \M
  },$$
  has a terminal object.
\end{lemma}
\begin{proof}
Let $\varepsilon:s\to t$ be an arrow of $\M$ such that $p(\varepsilon)=\earrow$.
Note that the projection
\[
\M_{/t/\varepsilon}\times_{\Delta^1}\{0\}\to\M_{/t}\times_{\Delta^1}\{0\}
\]
is an equivalence if and only if the composite
\[
\M_{/\varepsilon}\simeq\M_{/t/\varepsilon}\cong\M_{/t/\varepsilon}\times_{\Delta^1}\{0\}\to\M_{/t}\times_{\Delta^1}\{0\}
\]
is an equivalence (the isomorphism in the middle is a result of the fact that the projection $\M_{/t/\varepsilon}\to\Delta^1$ factors through $\{0\}\to\Delta^1$).
Moreover, we have isomorphisms $\Delta^1_{/1}\cong\Delta^1$, $\Delta^1_{/\earrow}\cong\{0\}$, and
\[
(\M_{/t}\times_{\Delta^1}\{0\})_{/\varepsilon}\cong\M_{/t/\varepsilon}\times_{\Delta^1}\{0\},
\]
which shows that
\[
(\M_{/t}\times_{\Delta^1}\{0\})_{/\varepsilon}\to\M_{/t}\times_{\Delta^1}\{0\}
\]
is an equivalence if and only if
\[
\M_{/\varepsilon}\to\M_{/t}\times_{\Delta^1_{/1}}\Delta^1_{/\earrow}
\]
is an equivalence.
The result now follows from the fact that $p$ is cocartesian if any $t\in\M$ with $p(t)=1$ is the target of a $p$-cartesian arrow $\varepsilon$.
\end{proof}

\begin{corollary}\label{maclanelemma}
  A functor $F:\C\to\D$ admits a right adjoint 
  if and only if, for every object $t\in\D$, the $\infty$-category 
  $\C_{/t}$, defined as the pullback
  \[
  \xymatrix{
  \C_{/t}\ar[r]\ar[d]\drpullback & \C\ar[d]^F\\
  \D_{/t}\ar[r] & \D,}
  \]
  has a terminal object.
\end{corollary}

\begin{proof}
Form the cocartesian fibration $p:\M\to\Delta^1$ associated to the left adjoint $F:\C\to D$; we must show that $p$ is also cartesian.
This follows from the previous lemma upon observing that $\C_{/t}$ is naturally identified with the pullback $\M_{/t}\times_{\Delta^1}\{0\}$.
\end{proof}

\begin{remark}
  We will refer to a right fibration corresponding 
  to a limit-preserving functor 
  $\C^\op\to\Gpd_\infty$ as a {\em sheaf},
  at least in the presentable situation.
  This is because the Yoneda embedding
  \[
  \C\too\Fun(\C^{\op},\Gpd_\infty)
  \]
  factors through the subcategory
  \[
  \Fun^\mathrm{lim}(\C^{\op},\Gpd_\infty)\subset\Fun(\C^{\op},\Gpd_\infty)
  \]
  of limit-preserving functors, and 
  the resulting fully faithful inclusion 
  $\C\to\Fun^\mathrm{lim}(\C^{\op},\Gpd_\infty)$ is an equivalence when $\C$ is presentable
  \cite[Proposition 5.5.2.2]{Lurie:HTT}.
  In other words, we are thinking of colimit diagrams $\colim_\alpha U_\alpha\simeq T$
  in $\C$ as specifying decompositions of $T$ into ``local'' pieces $U_\alpha$; if $\C$ is an $\infty$-topos, this is literally the canonical topology on $\C$.
\end{remark}

\begin{corollary}
  For any locally cartesian closed $\infty$-category $\C$ and for any pair of
  objects $X$ and $Y$ of $\C$, the right fibration $\sMap(X,Y)\to\C$ is a sheaf.
  \qed
\end{corollary}


\begin{blanko}{Spaces of equivalences.}
  Let $\C$ be an $\infty$-category and let $X$ and $Y$ be objects of $\C$. Define
  $$
  \Eq_{\C}(X,Y) := \Map_{\C^{\eq}}(X,Y) ,
  $$
  the space of maps from $X$ to $Y$ in the maximal subgroupoid $\C^{\eq}$ of $\C$.
  We have
  \[\pi_0 \Eq_\C(X,Y) \cong\Map_{\Ho(\C^{\eq})}(X,Y)\cong\operatorname{Iso}_{\Ho(\C)}(X,Y).
  \]
  It follows that
    \[
  \xymatrix{
  \Eq_\C(X,Y)\ar[r]\ar[d] \drpullback& \Map_\C(X,Y)\ar[d]\\
  \pi_0 \Eq_\C(X,Y)\ar[r] & \pi_0 \Map_\C(X,Y)
  }
    \]
    is always a pullback square, and that the inclusion
    $\Eq_\C(X,Y)\to\Map_\C(X,Y)$ is a monomorphism.

  We will often suppress the subscript $\C$ from the notation and simply write
  $\Eq(X,Y)$ for $\Eq_\C(X,Y)$.  The following lemma is immediate from the
  characterization of equivalences in $\C$ as those maps which become isomorphisms
  in $\Ho(\C)$.
\end{blanko}
\begin{lemma}
  Let $\C$ be an $\infty$-category and let $X$ and $Y$ be objects of $\C$.  Then a
  homotopy class of map $f\in\pi_0\Map(X,Y)$ lies in the subset $\pi_0\Eq(X,Y)$
  of homotopy classes of equivalences if and only if there exist homotopy class
  of map $g\in\pi_0\Map(Y,X)$ such that $g f\in\pi_0\Aut(X)\subset\pi_0\End(X)$
  and $f g\in\pi_0\Aut(Y)\subset\pi_0\End(Y)$.
  \qed
\end{lemma}

We now define the right fibration $\sEq(X,Y)$ of equivalences as the
full subfibration of $\sMap(X,Y) \to \C$ whose
fiber over $T$ is the subspace $\Eq_{/T}(X \times T,  Y \times T)\subset
\Map_{/T}(X \times T,  Y \times T)$.  
This is a subfibration since clearly basechange of an equivalence
is an equivalence.  Note that $\sEq(X,Y) \to \C$ has small
fibers since it is a subfibration of a right fibration $\sMap(X,Y) \to \C$ with small fibers, as the mapping spaces in $\C$ are small.

\begin{lemma}
There is a canonical inversion map $i:\Eq(X,Y)\to\Eq(Y,X)$, defined up 
to contractible ambiguity.
Moreover, this extends to an inversion on right fibrations
\[
\xymatrix{\sEq(X,Y)\ar[rr]^i\ar[rd] & & \sEq(Y,X)\ar[ld]\\
& \C &}
\]
such that for each object $T$ of $\C$, the diagram
\[
\xymatrix{
\Eq_{/T}(X\times T,Y\times T)\ar[r]\ar[d]^{(\id,i)} & \ast\ar[d]^{\id_{X\times T}}\\
\Eq_{/T}(X\times T,Y\times T)\times\Eq_{/T}(Y\times T,X\times T)\ar[r]\ar[d] & \Eq_{/T}(X\times T,X\times T)\ar[d]\\
\Map_{/T}(X\times T,Y\times T)\times\Map_{/T}(Y\times T,X\times T)\ar[r] & \Map_{/T}(X\times T,X\times T)}
\]
commutes.
\end{lemma}

\begin{proof}
  More generally, $(-)^{\op}$, regarded as an endoequivalence of the
  $\infty$-category of $\infty$-groupoids which squares to the identity, comes
  equipped with a natural transformation $i:(-)^{\op}\to\id$; moreover, for any
  $\infty$-groupoid $\G$ and any map $f:X\to Y$ in $\G$, the resulting involution
  $i_\G:\G^{\op}\to\G$ exhibits $i_\G(f)$ as an inverse of $f$.  See
  \cite[Section~4.2]{Gepner-Haugseng:1312} and \cite{DK84} for details.  Taking $\G=(\C_{/T})^{\eq}$, as $T$
  ranges through the objects of $\C$, we obtain a morphism of right fibrations
  $\sEq(X,Y)\to\sEq(Y,X)$ over $\C$ as claimed.
\end{proof}

\begin{proposition}\label{prop:eqrep}
  If $\C$ is cartesian closed, then 
  the right fibration $\sEq(X,Y) \to \C$ is 
  a sheaf. If $\C$ is also presentable, then the right fibration $\sEq(X,Y)\to\C$ is representable.
\end{proposition}
When the right fibration $\sEq(X,Y)\to\C$ is representable, we will denote by $\intEq(X,Y)$ the representing object.
Note that $\intEq(X,Y)$ is necessarily a subobject of $\intMap(X,Y)$, provided of course both exist.

\begin{proof}
  We already know that $\sEq(X,Y)\subset\sMap(X,Y)$ is a subpresheaf of a 
  sheaf, so it is 
  automatically $(-1)$-separated.  This is to say that given a colimit diagram $\colim_\alpha U_\alpha 
  \isopil T$ in $\C$, the resulting map
  $$
  \Eq_{/T}(X\times T,Y\times T)
  \longrightarrow 
  \lim_\alpha \Eq_{/U_\alpha}(X\times U_\alpha,Y\times U_\alpha)
  $$
  a monomorphism.  It remains to show that it is
  surjective on $\pi_0$.
  
  The involution $i:\sEq(X,Y)\to\sEq(Y,X)$ induces a map
  \[
  \lim_\alpha\Eq_{/U_\alpha}(X\times U_\alpha,Y\times U_\alpha)
  \to
  \lim_\alpha\Eq_{/U_\alpha}(Y\times U_\alpha,X\times U_\alpha)
  \]
  making the diagram
  \[
  \xymatrix{
  \lim_\alpha\Eq_{/U_\alpha}(X\times U_\alpha,Y\times U_\alpha)\ar[r]\ar[d]^{(\id,i)} & \ast\ar[d]^{\id_{X\times U_\alpha}}\\
 \lim_\alpha\Eq_{/U_\alpha}(X\times U_\alpha,Y\times U_\alpha)\times\lim_\alpha\Eq_{/U_\alpha}(Y\times U_\alpha,X\times U_\alpha)\ar[r]\ar[d] & \lim_\alpha\Eq_{/U_\alpha}(X\times U_\alpha,X\times U_\alpha)\ar[d]\\
  \lim_\alpha\Map_{/U_\alpha}(X\times U_\alpha,Y\times U_\alpha)\times\lim_\alpha\Map_{/U_\alpha}(Y\times U_\alpha,X\times U_\alpha)\ar[r]\ar[d]^{\simeq} & \lim_\alpha\Map_{/U_\alpha}(X\times U_\alpha,X\times U_\alpha)\ar[d]^{\simeq}\\
   \Map(X\times T,Y\times T)\times\Map(Y\times T, X\times T)\ar[r] & \Map_{/T}(X\times T,X\times T)}
  \]
  commute.
  The analogous diagram involving the other composition
  \[
  (i,\id): \lim_\alpha\Eq_{/U_\alpha}(X\times U_\alpha,Y\times U_\alpha)
  \too
  \lim_\alpha\Eq_{/U_\alpha}(Y\times U_\alpha,X\times U_\alpha)\times\lim_\alpha\Eq_{/U_\alpha}(X\times U_\alpha,Y\times U_\alpha)
  \]
  also commutes, so any $\tilde{f}\in\pi_0\lim_\alpha\Eq_{/U_\alpha}(X\times U_\alpha, Y\times U_\alpha)$ gives elements $f\in\pi_0\Map(X\times T,Y\times T)$ and $g\in\pi_0\Map(Y\times T, X\times T)$ such that $gf=\id_{X\times T}\in\pi_0\Map_{/T}(X\times T,X\times T)$ and $fg=\id_{Y\times T}\in\pi_0\Map_{/T}(Y\times T,Y\times T)$.
  Hence $f\in\pi_0\Eq_{/T}(X\times T,Y\times T)$, completing the proof.
\end{proof}

The following local version is immediate.

\begin{theorem}\label{thm:eqrep}
  Let $S$ be an object in a locally cartesian closed
  $\infty$-category $\C$.  Then for each pair of objects $X \to S$ and $Y \to S$ of
  $\C_{/S}$, the right fibration $\sEq_{/S}(X,Y) \to \C_{/S}$ is a sheaf. If in addition $\C$ is presentable then the right fibration $\sEq_{/S}(X,Y)\to\C_{/S}$ is representable.
  \qed
\end{theorem}
As before, we write $\intEq_{/S}(X,Y)\to S$ for the representing object in $\C_{/S}$, when it exists; it is a subobject of $\intMap_{/S}(X,Y)$.

\begin{blanko}{Fibrations of equivalences.}\label{blanko}
We conclude this section with some alternative characterizations of
fibrations of equivalences.
Let $\C$ be an $\infty$-category.
Recall that the maximal subgroupoid $\C^{\eq}$ of $\C$ can be 
identified with the inclusion of the space of objects of $\C$ into $\C$:
  \[
  \C^{\eq}\simeq\Map(\Delta^0,\C)\too\Fun(\Delta^0,\C)\simeq\C.
  \]
  The space of equivalences $\Map(\Delta^1[\earrow^{-1}],\C)$ of $\C$ is defined using the universal equivalence $\Delta^1[\earrow^{-1}]$, where $\earrow$ denotes the unique nonidentity arrow of $\Delta^1$; since the projection $\Delta^1[\earrow^{-1}]\to\Delta^0$ is a categorical equivalence,
  we see that
  \[
  \C^{\eq}\simeq\Map(\Delta^0,\C)\simeq\Map(\Delta^1[\earrow^{-1}],\C),
  \]
  which is to say that the space of objects of $\C$ is canonically equivalent to
  the space of equivalences of $\C$.
  Hence the $\infty$-categorical nerve $\Map(-,\C):\Delta^{\op}\to\Gpd_\infty$ is a {\em complete} Segal space in the sense of Rezk \cite{Rezk:MR1804411}.
  This notion of completeness is intimately related to the notion of univalence.
Note that if $\C$ is an $\infty$-category with pullbacks, then the inclusion of the target $\Delta^0\cong\Delta^{\{1\}}\to\Delta^1$ of $e$ induces the ``target'' cartesian fibration
\[
\Fun(\Delta^1,\C)\too\C,
\]
which comes equipped with a subfibration
\[
\Fun(\Delta^1[e^{-1}],\C)\subset\Fun(\Delta^1,\C)\too\C
\]
such that the composite is an equivalence.

  In order to compare various sheaves of equivalences
  which arise when considering univalent families in the next section, it will be 
  useful to record a straightforward analogue of this subfibration.
  Consider the inclusion 
  $\Lambda^2_2\to(\Lambda^2_2)^{\triangleleft}\cong\Delta^1\times\Delta^1$;
  it is the nerve of the poset inclusion
  \[
      \xymatrixrowsep{30pt}
    \xymatrixcolsep{30pt}
  \xymatrix @!=0pt {
  &&&&&&& (0,0)\ar[ld]_{\earrow_1}\ar[rd]^{\earrow_2}\ar[dd] &\\
  (0,1)\ar[rd] & & (1,0)\ar[ld]&& \subset && (0,1)\ar[rd] & & (1,0)\,\,. \ar[ld]\\
  &(1,1)&&&&&& (1,1) &}
  \]
\end{blanko}
\begin{lemma}\label{lem:wsibc}
Let $\C$ be an $\infty$-category will pullbacks.
Then the inclusion $\Lambda^2_2\to(\Lambda^2_2)^{\triangleleft}\cong\Delta^1\times\Delta^1$
induces a cartesian fibration
\[
\Fun(\Delta^1\times\Delta^1,\C)\too\Fun(\Lambda^2_2,\C),
\]
and the subfibration
\[
\Fun((\Delta^1\times\Delta^1)[\earrow_1^{-1},\earrow_2^{-1}],\C)
\subset\Fun(\Delta^1\times\Delta^1,\C)\too\Fun(\Lambda^2_2,\C)
\]
is a right fibration with fiber
over the object $X\to S\leftarrow Y$ 
is the $\infty$-groupoid $\Eq_{/S}(X\times S,Y\times S)$.
\end{lemma}
\begin{proof}
  The fiber over the object $\{X\to S\leftarrow Y\}$ is the $\infty$-category
  $\C_{/X\times_S Y}$, and the cartesian arrows over a map 
  $\{X'\to S'\leftarrow Y'\}\to\{X\to S\leftarrow Y\}$ are those 
  $Z'\to X'\times_{S'} Y'$ such that
  $Z'\simeq Z\times_{X\times_S Y}X'\times_{S'} Y'$.  The subfibration has fibers
  the full subcategory of the $\infty$-category $\C_{/X\times_S Y}$ consisting of
  those $Z\to X\times_S Y$ such that both composites $Z\to X$ and $Z\to Y$ are
  equivalences.  
  
  Since we are
  working entirely in $\C_{/S}$, we may suppose without loss of generality that
  $S$ is a terminal object of $\C$.  Using the equivalences
  \[
  \C^{\eq}\simeq\Map(\Delta^0,\C)\simeq\Map(\Delta^1[\earrow^{-1}],\C),
  \]
  we see that the fibers are equivalent to 
  $\{X\}\times_{\C^{\eq}}\C^{\eq}\times_{\C^{\eq}}\C^{\eq}\times_{\C^{\eq}}\{Y\}$,
  which in turn is equivalent to $\{X\}\times_{\C^{\eq}}\{Y\}$.
  This is the space of paths from $X$ to $Y$ in $\C^{\eq}$, i.e.~$\Eq_{\C}(X,Y)$.
  In particular, since the fibers of this fibration
  are $\infty$-groupoids, it is a right fibration.
\end{proof}

\section{Univalent families and local classes of maps}
\label{sec:loc-uni}

\newcommand{\fami}[2]{\Oooh_{#1}^{({#2})}}

\begin{blanko}{The equivalence sheaf of a family.}
  Let $\C$ be a presentable locally cartesian closed $\infty$-category.
  Given a map $p:X\to S$ we write
  \[
  \sEq_{/S}(X):=\sEq_{/S\times S}(p_1^*X,p_2^*X)\too\C_{/S\times S}
  \]
  for the right fibration of equivalences over $S\times S$ from $p_1^*X$ to
  $p_2^*X$, as in Section~\ref{sec:Eq}.  Its fiber over the object $(f,g):T\to
  S\times S$ of $\C_{/S\times S}$ is the space $\Eq_{/T}(f^*X,g^*X)$ of
  equivalences $f^*X\simeq g^*X$ over $T$.  By Theorem~\ref{thm:eqrep}, this
  right fibration is represented by the object of equivalences
  \[
  \intEq_{/S}(X):=\intEq_{/S\times S}(p_1^*X,p_2^*X) \too 
  S\times S .
  \]
\end{blanko}

\begin{blanko}{Univalence.}
The restriction of the internal equivalence object $\intEq_{/S}(X)\to S\times S$ 
along the diagonal $\delta:S\to S\times S$ comes equipped with a canonical 
section; namely, the one which sends the $T$-point $T\to S$ to the identity 
equivalence $\id_{X\times_S T}\in\Eq_{/T}(X\times_S T,X\times_S T)$.
In other words, we obtain a commutative diagram
$$
\xymatrix{
S \ar[rr]^\phi\ar[rd]_\delta && \intEq_{/S}(X) \ar[ld] \\
&S\times S&
}
$$
We say that $p:X\to S$ is a {\em univalent family} if
$
\phi:S \too \intEq_{/S}(X)
$
is an equivalence.\footnote{This definition is a generalization 
of the notion introduced by Voevodsky in the special case of simplicial sets,
which he proposed as an extra axiom for type theory;
cf.~Kapulkin--Lumsdaine--Voevodsky~\cite{Kapulkin-Lumsdaine:1211.2851}, 
\cite{Kapulkin-Lumsdaine-Voevodsky:1203.2553}.}
In order to analyze this notion we proceed to reformulate it in more abstract 
terms.
\end{blanko}

\begin{blanko}{Local classes of maps and universal families.}
  Let $\C$ be a presentable locally cartesian closed $\infty$-category.
  Adopting the notation of Lurie~\cite[Section~6.1]{Lurie:HTT},
  we write $\Oooh_\C = \Fun(\Delta^1,\C)$
  for the $\infty$-category of arrows in $\C$, always considered together with its target fibration
  $p:\Oooh_\C \to \C$ (evaluation at $1\in\Delta^1$).  Its fiber over an
  object $T$ is the slice $\infty$-category $\C_{/T}$.
  
  For a class of maps $\F$ in $\C$, assumed to be stable under 
  pullback,  we denote by $\Oooh_\C^\F$ the 
  full subcategory of $\Oooh_\C$ consisting of those maps which are in $\F$, and we denote by
  $\Oooh_\C^{(\F)}$ the category with the same objects as $\Oooh_\C^\F $ but 
  containing only the $p$-cartesian arrows (the pullback squares).
  The target fibration $\Oooh_\C^{(\F)} \to \C$ is a right fibration, 
  and its fiber over an object $T$ is the $\infty$-groupoid 
  $(\C_{/T})^{(\F)}$,
  the full $\infty$-groupoid of $\C_{/T}^{\eq}$ spanned by the arrows in $\F$.
  When $\F$ is the class of all maps in $\C$ we also write 
  $\Oooh_\C^{(\mathrm{all})}$ for the subcategory of all objects but only the $p$-cartesian arrows.
  Straightening\footnote{\cite[Chapter 2]{Lurie:HTT}} the right fibration $\Oooh_\C^{(\F)} \to \C$ 
  determines a (possibly large) presheaf
  \[
  F:\C^\op\to\widehat{\Gpd}_\infty\,\,.
  \]
  The class $\F$ (assumed stable under basechange) is called {\em
  local}\footnote{[HTT 6.1.3.8]} when 
  $F$ preserves small limits, i.e.~when it is a (possibly large) sheaf.  
  If $\F$ is a local class,
  then the sheaf $F:\C^{\op}\to\widehat{\Gpd}_\infty$ is representable
  if and only if it takes small values.\footnote{[HTT 6.1.6.3]}
  This can be ensured by a cardinal bound: for a regular cardinal $\kappa$, let
  $\F_\kappa$ denote the class of relatively $\kappa$-compact maps in 
  $\F$.\footnote{I.e.~maps $p\in \F$ such that pullback to a 
  $\kappa$-compact object is again $\kappa$-compact (cf.~[HTT 6.1.6.4]).}
  If $\kappa$ is sufficiently large then $\F_\kappa$ is again a local
  class,\footnote{What can go 
  wrong for smaller cardinals is that the class might not be closed under 
  pullback.  See [HTT 6.1.6.7] for details.} called a {{\em bounded local class}}.  In this case
  the corresponding sheaf $F_\kappa$ takes small values and hence
  is representable; 
  the representing object $S_{\F_\kappa}$ is called a 
  {{\em classifying object}} for
  the class.  It carries a {\em universal family}, i.e.~a map $X_{\F_\kappa}\to
  S_{\F_\kappa}$ which is terminal in $\fami{\C}{\F_\kappa}$.
  
Recall that the projection $\C_{/S}\too\C$
is a right fibration; its fiber over $T$ is $\Map(T,S)$,
which is to say that the corresponding
sheaf $\C^\mathrm{op}\to\mathrm{Gpd}_\infty$ 
is represented by $S$.  By the Yoneda lemma, 
a map $p:X\to S$ determines a morphism of right fibrations
\[
\kappa_p:\C_{/S}\too \Oooh_\C^\mathrm{(all)} .
\]
The image of a map $f:T\to S$ is a pullback $f^*X\to T$.  
Similarly, if $p:X\to
S$ belongs to a local class $\F$, then the projection
$(\Oooh_\C^{(\F)})_{/p} \to \C_{/S}$ is a trivial Kan fibration (by the universal 
property of the pullback); choosing a section (which amounts to
choosing pullbacks of $p$), and composing with the projection
$(\Oooh_\C^{(\F)})_{/p} \to \Oooh_\C^{(\F)}$ gives a 
morphism of right fibrations
$$
\xymatrix{\C_{/S} \ar[rr]\ar[rd] & &  \Oooh_\C^{(\F)}\ar[ld]\\
& \C &}
$$
which sends an object $f:T\to S$ to $f^*(p)$.
The fiber over $T$ of this map is
$\Map(T,S) \too (\C_{/T})^{(\F)}$.
Since the projection  $(\Oooh_\C^{(\F)})_{/p} \to \Oooh_\C^{(\F)}$
is an equivalence if and only if $p$ is terminal
  in $\Oooh_\C^{(\F)}$, we conclude:
\end{blanko}

\begin{lemma}\label{lemma:blc}
  Suppose a map $p:X\to S$ belongs to a local class $\F$.  
  Then the functor $\C_{/S} \to \Oooh_\C^{(\F)}$ which sends $f$ to $f^*(p)$ is an equivalence
if and only if $p$ is a universal family for $\F$, i.e.~is a terminal object
in $\Oooh_\C^{(\F)}$.
Moreover, in this case, $\F$ is a bounded local class.
\qed
\end{lemma}

\begin{blanko}{Examples.}
  The class of all equivalences is always local.  In a presentable locally
  cartesian closed $\infty$-category $\C$, the class of all maps is local if and
  only if $\C$ is an $\infty$-topos.\footnote{Combine \cite[Theorem 6.1.0.6]{Lurie:HTT}
  with \cite[Theorem 6.1.3.9]{Lurie:HTT}.
  } In this case there is for each sufficiently large
  regular cardinal $\kappa$ a universal family classifying relatively
  $\kappa$-compact maps.  These deserve to be called universes in $\C$.
  The class $\T_k$ of $k$-truncated maps is local in any $\infty$-topos
  (Corollary~\ref{ntr=loc}), and also in any $n$-topos for $k\leq n-2$ 
  (Corollary~\ref{k<=n-2}).
    
  An important example of this situation is when $\C$ is the $2$-topos of
  $1$-groupoids.  There is then a universal family of $0$-groupoids, i.e.~sets.
  This is a suitable setting for combinatorics, and is also exemplified by
  classical algebraic geometry where groupoid-valued sheaves are needed to
  provide universal families of schemes.
\end{blanko}

\begin{proposition}
  The diagonal functor $\Oooh_\C^{(\mathrm{all})} \too \Oooh_\C^{(\mathrm{all})} \times_\C 
  \Oooh_\C^{(\mathrm{all})}$ is equivalent to the right fibration
  \[
  \Fun((\Delta^1\times\Delta^1)[\earrow_1^{-1},\earrow_2^{-1}],\C)\too\Fun(\Lambda^2_2,\C)
  \]
  defined in \ref{blanko}.
\end{proposition}

\begin{proof}
We have a commutative square
\[
\xymatrix{
\Oooh_\C^{(\mathrm{all})}\ar[r]\ar[d] & \Fun((\Delta^1\times\Delta^1)[\earrow_1^{-1},\earrow_2^{-1}],\C)\ar[d]\\
\Oooh_\C^{(\mathrm{all})} \times_\C \Oooh_\C^{(\mathrm{all})}\ar[r] & \Fun(\Lambda^2_2,\C)}
\]
in which the top horizontal map is induced by restriction along the diagonal $\Delta^1\to\Delta^1\times\Delta^1$ and the bottom map is an equivalence.
By Lemma \ref{lem:wsibc}, the vertical fibers are all equivalent, so this square is a pullback and therefore the top map is also an equivalence.
\end{proof}

\begin{corollary}\label{cor:Eq-pbk}
  The sheaf of equivalences $\sEq_{/S} (X) \to \C_{/S\times 
  S}$ fits into the pullback square
  $$\xymatrix{
   \sEq_{/S}(X) \drpullback\ar[r]\ar[d] & \Oooh_\C^{(\mathrm{all})} 
   \ar[d]^{\mathrm{diag}} \\
   \C_{/S\times S} \ar[r]_-{\kappa_p\times \kappa_p} & \Oooh_\C^{(\mathrm{all})} \times_\C \Oooh_\C^{(\mathrm{all})}\,\,.
}$$
\qed
\end{corollary}

\begin{proposition}\label{prop:univalent=localclass}
  Let $p:X \to S$ be a map in a presentable locally cartesian closed
  $\infty$-category $\C$, and let $\F_p$ denote the class of maps obtained as
  pullbacks of $p$. 
  Then the following conditions are equivalent.
  \begin{enumerate}
    \item $p$ is univalent.
    
    \item The commutative square
   $$\xymatrix{
   \C_{/S} \ar[r]^-{\kappa_p}\ar[d]_{\mathrm{diag}} & \Oooh_\C^{(\mathrm{all})} 
   \ar[d]^{\mathrm{diag}} \\
   \C_{/S\times S} \ar[r]_-{\kappa_p\times \kappa_p} & \Oooh_\C^{(\mathrm{all})} \times_\C \Oooh_\C^{(\mathrm{all})}
}$$   
    is a pullback.
  
    \item The map 
    $\C_{/S} \to \Oooh_\C^{\mathrm{(all)}}$ is a monomorphism of right 
    fibrations, i.e.~is fiberwise a monomorphism.  That is,
    for every object $T\in\C$, the map
\begin{eqnarray*}
  \Map_\C(T,S) \longrightarrow \C_{/T}^{\eq}
  \end{eqnarray*}
  which sends $f$ to $f^*(p)$ is a monomorphism.
  
  \item 
  $\F_p$ is a bounded local class with $p$ a terminal object of $\Oooh_\C^{(\F_p)}$.
  \end{enumerate}
\end{proposition}

\begin{proof}
  We first prove (1)$\Leftrightarrow$(2).  The map 
  $$
\xymatrix{
S \ar[rr]^\phi\ar[rd] && \intEq_{/S}(X) \ar[ld] \\
&S\times S&
}
$$
represents the canonical map
$$
\xymatrix{
\C_{/S} \ar[rr]\ar[rd] && \sEq_{/S}(X) \ar[ld] \\
&\C_{/S\times S}&
}
$$
induced by the pullback property of $\sEq_{/S}(X)$ established in 
Corollary~\ref{cor:Eq-pbk}.  Hence $\phi$ is an equivalence if and only if
the square in (2) is a pullback.

  Next we prove (2)$\Leftrightarrow$(3).
  Condition (2) says that for every $T\in \C$, the square
\[
\xymatrix{
\Map(T,S)\ar[r]^{\kappa_p}\ar[d] \drpullback& \C_{/T}^{\eq}\ar[d]\\
\Map(T, S)\times\Map(T,S)\ar[r]_-{\kappa_p\times \kappa_p} & \C_{/T}^{\eq}\times\C_{/T}^{\eq}}
\]
is a pullback.
But this is to say that $\kappa_p$ is a monomorphism.

Finally (3)$\Leftrightarrow$(4).
Suppose $\Map(T,S) \to \C_{/T}^{\eq}$ is a monomorphism.
It is clear
that the essential image of $\Map(T,S) \to \C_{/T}^{\eq}$ is precisely 
$\C^{(\F_p)}_{/T}$,
so we have an equivalence $\Map(T,S) \isopil \C^{(\F_p)}_{/T}$.  
Using Lemma~\ref{lemma:blc}, this means that
$p$ is a terminal object in $\Oooh_\C^{(\F_p)}$, and in particular
$\F_p$ is a bounded local class, as it is bounded by any
regular cardinal
$\kappa$ bigger than all the fibers of $p$.  Conversely, if $\F_p$ is a bounded local
class with $p$ terminal in $\Oooh_\C^{(\F_p)}$, then
we have a monomorphism $\Map(T,S)\isopil \C^{(\F_p)}_{/T} \to 
\C_{/T}^{\eq}$.
\end{proof}

\begin{theorem}\label{poset}
  For any presentable locally cartesian closed $\infty$-category $\C$, the equivalence classes of univalent families $\C$ form a (possibly large) poset which is canonically equivalent to the poset of bounded local classes of maps in $\C$.
\end{theorem}

\begin{proof}
  That the (equivalence classes of) univalent families in $\C$ form
  a (possibly large) poset is to say that for any two univalent families
  $q$ and $p$, the space $\Map_{\Oooh_\C^{(\mathrm{all})}}(q,p)$ 
  is either empty or contractible.  But we know from 
  Proposition~\ref{prop:univalent=localclass}~(4) that $p$ is terminal in
  $\Oooh_\C^{(\F_p)}$, a full subcategory of $\Oooh_\C^{(\mathrm{all})}$;
  hence $\Map_{\Oooh_\C^{(\mathrm{all})}}(q,p)$ is contractible
  if $q$ can be obtained as a pullback of $p$, and clearly empty otherwise.
On the other hand, bounded
local classes form a poset under inclusion.
  Given a bounded local class of maps, the associated universal family $p$
  yields an equivalence
  \[\Map(T,S) \too \C_{/T}^{(\F)},\]
  and postcomposing with the monomorphism $\C_{/T}^{(\F)}\to \C_{/T}^{\eq}$
  yields a monomorphism, so $p$ is univalent, by Proposition~\ref{prop:univalent=localclass}.  It is clear that $\F$ is
  precisely the class of all pullbacks of $p$.  Conversely, given a univalent
  family $p$, the class $\F_p$ of all pullbacks of $p$ is a bounded local class,
  by the previous proposition, and it is clear that the universal family
  associated to $\F_p$ is equivalent to $p$.  It is clear that this one-to-one
  correspondence is inclusion-preserving.
\end{proof}

\begin{corollary}\label{subfamily-iff}
  Let $p':X'\to S'$ be the pullback of a univalent family $p:X\to S$ in $\C$
  along a map $m:S'\to S$.
  Then $p'$ is univalent if and only if $m$ is a monomorphism in $\C$.
\end{corollary}

\begin{proof}
  We use criterion (3) of Proposition~\ref{prop:univalent=localclass}:
  for each object $T\in \C$ we have the commutative diagram
  of spaces
  $$
  \xymatrix{
  \Map(T,S') \ar[d]_{m_!} \ar[r] & \C^{\eq}_{/T} \ar[d]^= \\
  \Map(T,S) \ar[r] & \C^{\eq}_{/T} \ .
  }$$
  The bottom map is mono by Proposition~\ref{prop:univalent=localclass}.
  Hence the top map is mono if and only 
  if $m_!$ is mono, which is to say that $m$ is a mono in $\C$.
\end{proof}

\begin{corollary}
  If $p': X'\to S'$ belongs to some bounded local class in a presentable locally
  cartesian closed $\infty$-category, then it is univalent if and only if it is the
  pullback of the associated universal family along a monomorphism.  \qed
\end{corollary}

Note that this univalence criterion is independent of which bounded local class $p'$ belongs 
to.  Note also that if $p'$ belongs to a local class $\F$, then it also belongs to
the bounded local class $\F_\kappa$, for $\kappa$ sufficiently large, and 
larger than any fiber of $p'$.
Finally note that if $\C$ is an $\infty$-topos, every $p'$ belongs to the local class of 
all maps, so we get a complete classification of univalent families in this 
case.

\begin{theorem}\label{thm:induce-along-R}
  Let $L: \P\to\C \subset \P$ be a accessible locally cartesian
  localization of a presentable locally cartesian closed $\infty$-category
  $\P$.  Then a map $p:X\to S$ in $\C$ is a univalent family in $\C$ if
   and only if it is also a univalent
  family in $\P$.
\end{theorem}

\begin{proof}
  Let $T$ be an object of $\P$.  In the diagram
  $$
  \xymatrix{
  \Map_{\C}(LT,S) \ar[r] \ar[d]_\simeq & \C_{/LT}^{\eq}\ar[d]\\
  \Map_{\P}(T,S) \ar[r]  & \P_{/T}^{\eq} ,
  }$$
  the left-hand map is an equivalence by adjunction.
  Since $L$ is a \loccart\ localization, $\C_{/LT} \to \P_{/T}$ 
  is always fully faithful by Lemma~\ref{lem:sliceadjoint} and therefore
  $\C^{\eq}_{/LT}\too\P^{\eq}_{/T}$ is mono.  Hence the top map is mono if and 
  only if the bottom map is mono,
  and we conclude again by criterion (3) of
  Proposition~\ref{prop:univalent=localclass}.
\end{proof}
 
\begin{blanko}{Colocalizations.}
  Although at the moment we do not exploit this in any way,
  we mention that there is another class of functors which preserve univalent 
  families, the colocalizations.
\end{blanko}

\begin{proposition}
  Let $R:\P\to\C\subset \P$ be a colocalization (with fully faithful left adjoint 
$F: \C \to\P$) between presentable locally cartesian closed $\infty$-categories.  If $p:X\to S$ is univalent in $\P$ then $R(p):R(X)\to R(S)$ is univalent in 
$\C$.
\end{proposition}

\begin{proof}
  For any  $T\in\C$,
  consider the commutative diagram
  \[
  \xymatrix{
  \Map_{\C}(T,RS)\ar[r] & {\C}_{/T}^{\eq}\ar[d]\\
  \Map_{\P}(FT,S)\ar[r]\ar[u]^\simeq & \P^{\eq}_{/FT}   ,
  }
  \]
  where the left-hand vertical map is an equivalence by adjunction,
  and where the horizontal maps are given by pulling back $R(p)$ and $p$,
  respectively.  That the diagram actually commutes relies on the fact
  that the unit $\id\Rightarrow RF$ for the adjunction is an equivalence.
  The right-hand vertical map is mono since $F$ and hence
  $F : {\C}_{/T} \to \P_{/FT}$ are fully faithful.
  Since $p:X\to S$ is univalent in $\P$, the bottom horizontal map
  is mono, and hence the top horizontal map is also mono, which
  is to say that $R(p): R(X) \to R(S)$ is univalent.
\end{proof}

\section{Factorization systems and truncation}
\label{sec:fact}

We refer to a factorization system\footnote{\cite[Section~5.2.8]{Lurie:HTT}}
on an $\infty$-category $\C$, in which $\E$ is the left orthogonal class 
and $\F$ is the right orthogonal class, by writing $(\E,\F)$.
We will often require that our factorization system $(\E,\F)$ 
is a {\em stable} factorization system, i.e.~that it is stable under basechange; this is automatic for the right orthogonal class $\F$, so this amounts to saying that pullbacks of arrows in the 
left orthogonal class $\E$ stay in $\E$.

Our first aim is to establish that $n$-connected maps and $n$-truncated
maps form a stable factorization system, in any presentable locally
cartesian closed $\infty$-category (Propositions~\ref{conntrunc} and
\ref{nconnstable}).  This will be our main example.

\begin{blanko}{Connectedness.}
  Let $\C$ be an $\infty$-category with finite limits, and let $n\geq -2$.
  {Recall (from [HTT 5.5.6.8]) that a map $f:X\to Y$ in $\C$ is 
  {\em $n$-truncated} when for every object $Z$, the induced map 
  $\Map(Z,X) \to \Map(Z,Y)$ is an $n$-truncated map of spaces, i.e.~has
  $n$-truncated fibers.}
  A map $p:A\to B$ in $\C$
  is called {\em $n$-connected}\,\footnote{Here we follow 
  Joyal~\cite{Joyal:CRM}. 
  When $\C$ is an $\infty$-topos, this notion agrees with
  Lurie's $(n+1)$-connective ([HTT 6.5.1.10]).
  For more general $\C$, it would 
     perhaps be more appropriate to say ``strong $n$-connected'',
     in analogy with the notion of strong
     epimorphism in ordinary category theory.
  }
  if it is left orthogonal to every $n$-truncated
  map in $\C$.  Equivalently, since $n$-truncated maps are stable under 
  basechange,\footnote{[HTT 5.5.6.12]} 
  it is enough to have the unique lifting property
  for all squares of the form
  $$\xymatrix{
     A \ar[r]\ar[d]_p & X \ar[d]^f \\
     B \ar[r]_= & B
  }$$
  with $f$ $n$-truncated. 
  Precisely, this lifting property says that for all $n$-truncated $f:X\to B$,
  the natural map
  $$
  \Map_{/B}(\id_B,f)  \to \Map_{/B}(p,f) 
  $$
  is an equivalence.
  
  An object $A$ is called {\em $n$-connected} if the map $A \to 1$ is 
  $n$-connected.  By the previous remark, this is equivalent to the condition
  that for every $n$-truncated object $X$, the natural map 
  $
  \Map_\C(1,X) \to \Map_\C(A,X)
  $
  is an equivalence.
  
  Recall from [HTT 5.5.6.10] that a map $f:X\to B$ is
  $n$-truncated in $\C$ if and only if
  it is an $n$-truncated object of $\C_{/B}$. The previous observations 
  immediately imply the corresponding result for $n$-connected maps:
\end{blanko}

\begin{lemma}
  In an $\infty$-category $\C$ with finite limits, a
  map $p:A\to B$ is $n$-connected in $\C$ if and only if $p$ is an 
  $n$-connected object in $\C_{/B}$.
  \qed
\end{lemma}

\begin{lemma}\label{ladjnconn}
  \footnote{Generalizing [HTT 6.5.1.16 (4)] which states that the
  left adjoint of a geometric morphism between $\infty$-topoi preserves 
  $n$-connected maps.}
  Any left adjoint between $\infty$-categories with finite limits preserves 
  $n$-connected maps.
\end{lemma}
\begin{proof}
  This follows from orthogonality since the right adjoint preserves 
  $n$-truncated maps, as does every left-exact functor (between 
  $\infty$-categories with finite limits), by [HTT 5.5.6.16].
\end{proof}

The notion of $n$-connected map is most interesting in the case where $\C$ is 
presentable, so that the inclusion functor $i:\tau_{\leq n} \C \to \C$ has a left adjoint
$n$-truncation functor $\tau_{\leq n} : \C \to \tau_{\leq n}\C$.\footnote{[HTT 5.5.6.18]}

\begin{lemma}\label{connected-taueq}
  Let $B$ be an $n$-truncated object in a presentable $\infty$-category $\C$.  Then
  a map $p:A \to B$ is $n$-connected if and only if $\tau_{\leq n}(p)$ is an
  equivalence.
\end{lemma}

\begin{proof}
  If $p$ is $n$-connected, then $\tau_{\leq n}(p)$ is $n$-connected by
  Lemma~\ref{ladjnconn}, and is clearly also $n$-truncated, hence is an
  equivalence.  For the converse implication, assuming that $\tau_{\leq n}(p)$
  is an equivalence, we must establish that $p \bot f$ for every $n$-truncated
  map $f:X\to B$.  Since $B$ is $n$-truncated, also $X$ is
  $n$-truncated,\footnote{by [HTT 5.5.6.14]} so $f$ is actually a map in
  $\tau_{\leq n}\C$, say $f=i(f)$.  But then the desired orthogonality $p \bot
  i(f)$ is equivalent to $\tau_{\leq n}(p) \bot f$, which is true since an
  equivalence is orthogonal to any map.
\end{proof}

\begin{corollary}\label{connected-terminal}
  \footnote{Cf.~[HTT~6.5.1.12] for the case where $\C$ is an $\infty$-topos.}
  In a presentable $\infty$-category $\C$, an object $A$
  is $n$-connected if and only if its truncation $\tau_{\leq n}A$ is terminal.
  \qed
\end{corollary}

\begin{proposition}\label{conntrunc}
  \footnote{Cf.~[HTT 5.2.8.16] for the case where $\C$ is an $\infty$-topos.}
In any presentable $\infty$-category $\C$, and for each $n\geq -2$, there is
  a factorization system formed by the $n$-connected maps and the
  $n$-truncated maps.
\end{proposition}
\begin{proof}
  By construction these two classes are orthogonal, and clearly closed under 
  equivalences.  It therefore only remains to establish that every map $f:A \to 
  B$ admits a factorization as an $n$-connected map followed by an $n$-truncated map.
  This factorization is simply
  $$\xymatrix{
     A \ar[rr]^f\ar[rd]_{\eta_f} && B  \\
     & \tau^B_{\leq n}(A) \ar[ru]_{\tau^B_{\leq n}(f)} & 
  }$$
  where $\tau^B_{\leq n}$ denotes the $n$-truncation in $\C_{/B}$, and $\eta$ is
  the unit for that adjunction.  The map $\tau^B_{\leq n}(f)$ is $n$-truncated
  by construction; the key point is to see that $\eta_f$ is $n$-connected.
  Since clearly 
  $\tau^B_{\leq n}(\eta_f):\tau^B_{\leq n}(A)\to\tau^B_{\leq n}\tau^B_{\leq n}(A)$ 
  is an equivalence in $\C_{/B}$, it
  follows from Lemma~\ref{connected-taueq} that $\eta_f$ is $n$-connected
  as a map in $\C_{/B}$.  Finally note that the projection functor $\C_{/B} \to
  \C$ preserves $n$-connected maps since it is a left adjoint (\ref{ladjnconn}), hence $\eta_f$ is
  $n$-connected also as a map in $\C$.
\end{proof}

\begin{corollary}\label{A->LA}
  For each $n\geq -2$, and for every object $A$ in a presentable category, the unit
  $\eta: A \to \tau_{\leq n} (A)$ is $n$-connected.
  \qed
\end{corollary}

\begin{proposition}\label{nconnstable}
  \!\!\footnote{Cf.~\cite[Proposition~6.5.1.16 (6)]{Lurie:HTT}
  for the case where $\C$ is an $\infty$-topos.}
  In a presentable locally cartesian closed $\infty$-category $\C$, the class of $n$-connected
  maps is stable under basechange.
\end{proposition}

\begin{proof}
  The statement is that for any map $g:B'\to B$, the pullback
  functor $g^* : \C_{/B} \to \C_{/B'}$ preserves $n$-connected objects.
  But since $g^*$ is assumed to be a 
  left adjoint, it preserves $n$-connected maps by Lemma~\ref{ladjnconn},
  and since it also preserves terminal objects, it therefore preserves
  $n$-connected objects.
\end{proof}

Having established our main example of a stable factorization system, we
proceed to some general results.

\begin{lemma}
  Let $\C$ be a presentable locally cartesian closed
  $\infty$-category $\C$ in which sums are disjoint,
  and let $(\E,\F)$ be a stable factorization system in $\C$. 
  Then $\F$ is closed under small sums.
\end{lemma}
\begin{proof}
  Let $\{f_i : X_i \to S_i\}$ be a small set of maps in $\F$, and let $f:X\to S$ be 
  their sum in $\C$.  Since sums are universal, for each $i$ we have 
  a pullback square
  $$\xymatrix{
  X_i\drpullback \ar[d]_{f_i} \ar[r] & X \ar[d]^f \\
  S_i \ar[r] & S
  }$$
  where the horizontal maps are the sum inclusions.
  Now consider a 
  commutative square
  \begin{equation}\label{sqem}
  \xymatrix{
  X' \ar[d]_{f'} \ar[r] & X \ar[d]^f \\
  S' \ar[r]_s & S
  }  \end{equation}
  where $f'$ belongs to $\E$. 
  We must show that the space of diagonal fillers is contractible,
  which  more formally  amounts to (cf.~dual of [HTT 5.2.8.3])
  establishing that the map
  \begin{equation}\label{pref'}
  \Map_{/S}(s,f) \too \Map_{/S}(sf',f)
  \end{equation}
  given by pre-composition with $f'$ is a homotopy equivalence.
  Pulling back the square to $S_i$ yields
  $$\xymatrix{
  X'_i \ar[d]_{f'_i} \ar[r] & X_i \ar[d]^{f_i} \\
  S'_i \ar[r]_{s_i} & S_i ,
  }$$
  and the sum of all these squares gives back the original square \eqref{sqem},
  since sums are disjoint.  Hence the map in \eqref{pref'} is
  the product (indexed by $i$) of all the maps
  $$
    \Map_{/S_i}(s_i,f_i) \too \Map_{/S_i}(s_if'_i,f_i)
$$
given by precomposition with $f'_i$.
But each of these maps is an equivalence since $f'_i \bot f_i$ by 
basechange stability of the factorization system.
Thus the product map is an equivalence too, as required.
\end{proof}

\begin{theorem}\label{M=local}\hspace*{-3pt}\footnote{A version of this result 
  for ordinary categories can be found in \cite[Remark 6.12]{Carboni-Janelidze-Kelly-Pare}}
  Let $\X$ be an $\infty$-topos and let $(\E,\F)$ be a 
  factorization system in $\X$ which is stable under basechange.
  Then $\F$ is a local class.
\end{theorem}

\begin{proof}\hspace*{-3pt}\footnote{generalizing [HTT 6.5.2.22]}
  Since $\F$ is closed under sums by the previous lemma, we can apply
  [HTT 6.2.3.14]: given a pullback diagram
  $$\xymatrix{
  X_0 \ar[d]_{f_0}\ar[r] \drpullback & X \ar[d]^f
  \\
  Y_0 \ar[r]_e & Y}$$
  in which $e$ is an effective epi, we need to show that if $f_0$ belongs to $\F$
  then already $f$ belongs to $\F$.
  Let $Y_\bullet$ be the \v Cech nerve of $Y_0 \to Y$, and let $X_\bullet$
  be the \v Cech nerve of $X_0 \to X$ (the pullback of $X$ to $Y_\bullet$).
  At each level $n$, since $\F$ is stable under basechange, 
  the map $X_n \to Y_n$ is in $\F$.
  We need to check
  that $f$ 
  is right orthogonal to any map in $\E$.  This condition can be 
  expressed\footnote{dual of [HTT 5.2.8.3]}
  by saying that for any map $p: A\to B$ in $\E$, the map given by 
  precomposition with $p$
  $$
  \Map_{\X_{/Y}}( B,  X) \too 
  \Map_{\X_{/Y}}( A,  X)
  $$
  is a homotopy equivalence.  Some abuse of notation is involved here: we 
  assume that the map $p: A \to B$ is over $Y$, so as to form the square we 
  need to fill, and the objects $B$, $X$ and $Y$ are regarded as objects over 
  $Y$.
  
  Now each of these mapping spaces can be obtained as the totalization of a 
  cosimplicial space (e.g.~for $B$)
  which in degree $n$ is given by $\Map_{\X_{/Y_n}}( B_n,  X_n)$,
  where $ A_\bullet$ and $ B_\bullet$ are the objects
  pulled back to the \v Cech nerve of  $Y_0 \to Y$.  So it is enough to
  show that for each $n$, the map
    $$
  \Map_{\X_{/Y_n}}( B_n,  X_n) \too 
  \Map_{\X_{/Y_n}}( A_n,  X_n)
  $$
  is a homotopy equivalence.  But this follows from the assumption:
  we have already remarked that each of the maps $f_n$ is in $\F$,
  and since the class $\E$ is assumed to be stable under basechange,
  also each of the maps $p_n$ is in $\E$. So by orthogonality
  of $p_n$ with $f_n$ we do have the required homotopy equivalence.
\end{proof}

\begin{corollary}\label{ntr=loc}
  In an $\infty$-topos, the class of $n$-truncated maps is local ($-2\leq n < 
  \infty$).
  The class of hypercomplete maps is also local
  (and of course the class of all maps is local).
\end{corollary}

\noindent
{Indeed, each of the classes listed is the right orthogonal class of a 
stable factorization system, the class of $n$-truncated maps by 
Proposition~\ref{nconnstable}.  Regarding hypercomplete maps, recall
that a map is {\em $\infty$-connected} if it is $n$-connected for all $n\geq -2$.
These form a basechange stable and saturated class of maps, and constitute
therefore\footnote{[HTT 5.5.5.7]} the left class of a factorization system,
whose right class is by definition the class of hypercomplete maps ([HTT~6.5.2.21]).}

\begin{blanko}{$\Q$-quasi-left-exact localization.}
  A localization $L:\P\to\C\subset\P$ is called is {\em $\Q$-quasi-left-exact}, with
  respect to a class of maps $\Q$ in $\C$ (closed under equivalences), when for
  each pullback square in $\P$
$$\xymatrix{
X' \drpullback \ar[r]\ar[d] &X \ar[d] \\
Y' \ar[r] & Y
}$$
the natural comparison map $L(X') \to L(Y') \times_{L(Y)} L(X)$ 
belongs to $\Q$.
\end{blanko}

A map in a presentable $\infty$-category $\C$ is called an 
{\em $n$-gerbe}\footnote{Joyal~\cite{Joyal:CRM}, p.181}
when it is simultaneously
$(n-1)$-connected
and $n$-truncated.  (Intuitively, its only nonzero
relative homotopy group is $\pi_n$.)
Let $\K_n$ denote the class of  $n$-gerbes in $\C$.
Note that $\K_n$ also restricts to a class in the full
subcategory $\tau_{\leq n}\C\subset\C$ of $n$-truncated objects.

\begin{lemma}\label{K}
For an $\infty$-topos $\P$, the truncation
functor $L:=\tau_{\leq n} : \P\to\tau_{\leq n}\P$ is $\K_n$-quasi-left-exact.
\end{lemma}

\begin{proof}
With reference to a pullback square as above,
we must show that the comparison map $L(Y'\times_Y X) \to LY' \times_{LY} LX$ 
belongs to $\K_n$.
But this map is $Lu$, where
$u : Y'\times_Y X \to LY' \times_{LY} LX$ is the composite
$$
Y'\times_Y X \stackrel{u_1}\too Y' \times_{LY} X 
\stackrel{u_2}\too LY' \times_{LY} X
\stackrel{u_3} \too LY' \times_{LY} LX .
$$
The map $u_3$ is the pullback of $X \to LX$, which is $n$-connected 
by Corollary~\ref{A->LA}, 
and since $n$-connected maps are stable under pullback 
by Proposition~\ref{nconnstable},
$u_3$ is $n$-connected as well,
and in particular $(n-1)$-connected.  The same argument applies to $u_2$, 
which is the pullback of $Y' \to LY'$.  Finally $u_1$ sits in the diagram
$$\xymatrix{
Y'\times_Y X \drpullback \ar[r]^{u_1} \ar[d] & Y'\times_{LY} X \ar[d] \\
Y \ar[r]_-\delta & Y\times_{LY} Y.
}$$
Since $Y \to LY$ is $n$-connected, the diagonal $\delta$ is
$(n-1)$-connected\footnote{[HTT 6.5.1.18]}
and hence the pullback $u_1$ is $(n-1)$-connected.
Altogether $u= u_3 \circ u_2 \circ u_1$ is $(n-1)$-connected.
Since $L$ preserves $(n-1)$-connected maps 
by Lemma~\ref{ladjnconn},
also $Lu$ is $(n-1)$-connected,
and it is also $n$-truncated.  Hence it is in $\K_n$ as asserted.
\end{proof}

\begin{lemma}\label{Q}
  Let $\U$ be a local class in a presentable locally cartesian closed
  $\infty$-category $\P$ and let $L:\P\to\C\subset\P$ be a 
  $\Q$-quasi-left-exact, accessible, and \loccart\
  localization for a class $\Q$ in $\C$.  
  Consider a class $\F$ in $\C$ such that $L(\U) \subset \F$ and $G(\F) 
  \subset \U$ (here $G$ denotes a right adjoint to $L$).
  If $\Q$ is left orthogonal to $\F$
  then $\F$ is a local class in $\C$.
\end{lemma}

When $\Q$ is the class of equivalences, we are just talking left-exact
localization.
The lemma implies in this case that left-exact localizations
preserve the property that all maps form a local class, and actually the proof is only a slight modification of the proof of this 
result in \cite[Proposition~6.1.3.10]{Lurie:HTT}.
For more general $\Q$, the class of all maps cannot stay local, but the lemma
says that more restricted classes can, provided they are right orthogonal to $\Q$.

\begin{proof}
  By \cite[Lemma 6.1.3.7]{Lurie:HTT}, we must show that, for each diagram $\sigma: \Lambda_0^2 \to \Oooh_\C$, denoted $g\overset{\alpha}{\longleftarrow} f\overset{\beta}{\longrightarrow} h$,
    in which $f$, $g$, and  $h$ belong to $\F$ and $\alpha$ and $\beta$ are 
  cartesian transformations, there is a colimit diagram
  $\overline \sigma: \Lambda_0^2{}^\triangleright \to \Oooh_\C$
  $$
  \xymatrix{f \ar[r]^\beta \ar[d]_\alpha & h \ar[d]^{\alpha'} \\
  g \ar[r]_{\beta'} & p
  }$$
  such that $\alpha'$ and $\beta'$ are cartesian and $p$ is again 
  in $\F$.  We can assume that $\sigma= L\circ \sigma'$ for some diagram 
  $\sigma': \Lambda_0^2 \to \Oooh_\P$ equivalent to $G \circ \sigma$.  As $G$ 
  is a right adjoint, $G(\alpha)$ and $G(\beta)$ are cartesian, and by 
  assumption $G(f)$, $G(g)$ and $G(h)$ are in $\U$, so since $\U$ is a local 
  class, there is a colimit diagram
  $\overline \sigma': \Lambda_0^2{}^\triangleright \to \Oooh_\P$
  $$
  \xymatrix{
  G(f) \ar[r]^{G(\beta)} \ar[d]_{G(\alpha)} & G(h) \ar[d]^{\gamma} \\
  G(g) \ar[r]_{\delta} & q
  }$$
  in which $\gamma$ and $\delta$ are cartesian and $q$ is in 
  $\U$. Now apply $L$ to get a pushout square
  $$
  \xymatrix{
  f \ar[r]^{\beta} \ar[d]_{\alpha} & h \ar[d]^{L\gamma} \\
  g \ar[r]_-{L\delta} & L(q) .
  }$$
  By assumption, $L(q) \in \F$, so it remains to check that $L(\gamma)$ and 
  $L(\delta)$
  are cartesian.  Let's look at $L(\delta)$: its components
  are  diagrams
  $$\xymatrix{
  A' \ar[r]^{L(\delta)_0} \ar[d]_g & A \ar[d]^{L(q)} \\
  B' \ar[r]_{L(\delta)_1}
  & B
  }$$
  in which $g$ and $L(q)$ are in $\F$.  A priori this square might not be a 
  pullback, but the comparison map $u$
  $$
  \xymatrix{
  A' \ar[r]^-u \ar[rd]_g & B' \times_B A \ar[d]_\pi \ar[r] \drpullback & A 
  \ar[d]^{Lq}\\
  & B' \ar[r]_{L(\delta)_1} & B }$$
  belongs to $\Q$.  Now consider the class $\overline \F = \Q^\bot$.
  It is stable under pullback, hence the map $\pi$ belongs to 
  $\overline\F$ (since $L(q)$ does), and $g$ also belongs to $\F\subset 
  \overline \F$ by assumption.
  By the left-cancellation property of right orthogonal 
  classes\footnote{\label{lcanc}By the left-cancellation property we mean: if in a triangle 
  $gf=h$ both $g$ and $h$ belong to a right orthogonal class, then so does $f$, 
  cf.~[HTT 5.2.8.6].} 
  it follows that also $u$ belongs to $\overline\F$.
  Since it also belongs to $\Q$ it must therefore be an equivalence, because the intersection of orthogonal classes is necessarily contained in the class of equivalences.
  Therefore $L(\delta)$ is cartesian.  (In conclusion, 
  although $L$ might not preserve all pullbacks, it does preserve just enough
  pullbacks to preserve locality).
\end{proof}

\begin{corollary}\label{k<=n-2}
  In an $n$-topos, the class $\T_k$ of $k$-truncated maps is local 
  for all $-2 \leq k \leq n-2$.
\end{corollary}

\begin{proof}
  Every $n$-topos $\X$ arises as the $(n-1)$-truncation of an 
  $\infty$-topos,\footnote{[HTT 6.4.1.5]} say
  $\X=\tau_{\leq n-1}\Y$, where $\Y$ is an $\infty$-topos.
  The truncation functor $L := \tau_{\leq n-1} : \Y 
  \to \X$ is $\K_{n-1}$-quasi-left-exact by Lemma~\ref{K}, 
  and accessible and \loccart\ by Lemma~\ref{taulcc}.
  We have $\K_{n-1} \bot \T_k$ in $\X$ (since the maps in $\K_{n-1}$ are $(n-2)$-connected
  and hence $k$-connected).
  Now we can apply Lemma~\ref{Q} (with $\U=\T_k$ the class of $k$-truncated maps in $\Y$,
  and $\F=\T_k$ be the class of $k$-truncated maps in $\X$)
  to conclude that $\T_k$ is a local class in $\X$.
\end{proof}

\begin{blanko}{Stability properties of local classes.}\label{stabcond}
  As a necessary condition for a univalent family to serve as a universe for a
  type theory, the
  corresponding (bounded) local class should be stable under the type-forming
  operations: dependent sums (lowershriek) and dependent products (lowerstar),
  as well as identity types.  Every local class is stable under pullback, by
  definition.  If a local class $\F$ is closed under composition, then it is
  clearly stable under dependent sums (lowershriek), but only along maps in
  $\F$.
\end{blanko}

\begin{lemma}\label{closedunderPi}
  If $\F$ is the right orthogonal class of a stable factorization system
  then it is closed under dependent products.  Moreover, if
  $\kappa$ is a strongly inaccessible cardinal, then $\F_\kappa$ is
  closed under dependent products along maps in
  $\F_\kappa$.
\end{lemma}

\begin{proof}
  Let $p$ be any map, and suppose $f$ is a map in $\F$.  We need to check that
  $p_*(f)$ belongs to $\F$ too.  For this we need to check $e\bot p_* f$ for all
  maps $e$ in the left class.  But we have $e\bot p_* f$ if and only if $p^* e
  \bot f$, and this last statement is true for all $e$ because $p^* e$ then
  belongs to the left class by stability.  Concerning $\F_\kappa$, it remains
  to observe that if both $f$ and $p$ are relatively $\kappa$-compact, then
  also $p_* f $ is so; this is ensured if $\kappa$  is strongly inaccessible 
  (see [HTT 5.4.2.9]).
\end{proof}

\begin{lemma}\label{closedunderdiag}
  If $\F$ is the right orthogonal class of a factorization system,
  then it is
  closed under taking diagonals: if $f:X\to Y$ belongs to $\F$ then so does
  $\delta_f: X \to X \times_Y X$.
\end{lemma}

\begin{proof}
  If $f:X \to Y$ belongs to $\F$ then so does the (first) projection $X\times_Y
  X \to X$ since it is the pullback of $f$, and right orthogonal
  classes are stable under
  pullback.  But the diagonal is a section of the projection, so since $\F$ is a
  right orthogonal class, the diagonal belongs to $\F$ 
  too.
\end{proof}

\section{\Qi topoi}
\label{sec:quasi}

In this section we study certain large univalent families which exist outside
the realm of $\infty$-topoi or $n$-topoi.  To do so, we introduce a general notion
of \qi topos, which generalizes the classical notion not only by replacing
ordinary topoi by higher topoi, but also by replacing mono by the right
orthogonal class of more general factorization systems.  Although this notion is
interesting in its own right, it is beyond the scope of the present paper to
develop the theory of \qi topoi systematically, so our treatment is somewhat ad
hoc and aims only to provide new examples of univalent families.

\begin{blanko}{From classical quasitopoi to \qi topoi.}
  There are several possible characterizations of (Grothen\-dieck) quasitopoi.  
  An elegant
  systematic treatment was provided recently by Garner and
  Lack~\cite{Garner-Lack:1106.5331}.  As our starting point for generalization to
  the $\infty$-setting, we take the viewpoint of separated presheaves, and proceed
  to relate the notion to locally cartesian localizations, in keeping with our
  general philosophy.
  
  Classically, a (Grothendieck) quasitopos can be described as a full
  subcategory $\Q$ of a (Grothendieck) topos $\P$ consisting of those objects 
  which are separated with respect to a Lawvere--Tierney topology 
  $\tau$ on $\P$. Recall (see, for instance, \cite[Section V.2]{MacLane-Moerdijk}) that an object $X$ of $\P$ is a sheaf for the topology $\tau$ if, given a $\tau$-covering sieve $i:S\to T$ in $\P$,
  the restriction $i^*:\Hom(T,X) \to \Hom(S,X)$ is an isomorphism. The full subcategory $\X$ of $\P$ consisting of the $\tau$-sheaves is again a topos, namely the left-exact localization of $\P$ obtained by applying the $\tau$-sheafification functor.
  If instead we only require that $i^*:\Hom(T,X)\to\Hom(S,X)$ is a monomorphism, then we arrive at the notion of a separated presheaf.
  One can prove that this condition is equivalent to
  asking the unit for the $\tau$-sheafification functor to be a monomorphism.
  The quasitopos $\Q$ is again a localization of $\P$, but this localization is no longer left-exact: in general, it only preserves monomorphisms and pullbacks over local objects.
  
  In the setting of $\infty$-topoi, it is no longer true that every left-exact
  localization arises from a Grothendieck topology (the ones that do are called
  {\em topological localizations}\footnote{\cite[Section~6.2.2]{Lurie:HTT}}),
  but the formulation of separatedness in terms of monomorphic unit works well,
  and leads to a working notion of \qi topos.  However, while the notion of
  monomorphism plays a distinguished role in $1$-topos theory (as the class of
  maps for which there is a classifier), in higher topoi there are more general
  such classes of maps, viz.~the class of $n$-truncated maps (or indeed 
  the class of all maps), and 
  one may envisage a distinct notion of \qi topos for
  each of the factorization systems ($n$-connected, $n$-truncated).
  In fact, one can
  relate the definition to more general stable 
  factorization systems, as we now proceed to do.

  Before coming to the definition, let us remark that
  in the case where $\P$ and $\X$ are both topological localizations of an ambient presheaf $\infty$-topos, with one topology finer than other, and $L:\P \to \X \subset \P$ is the induced left-exact localization, the
  notion of separatedness in terms of coverings makes sense. Furthermore, it is not
  difficult to prove that it is equivalent to being separated in the sense of
  monomorphic unit for sheafification, the essential point in the proof being
  that covering sieves form a filtered system and that filtered colimits
  preserve monomorphisms.  In the abstract setting with reference to a 
  right orthogonal class $\F$, it is also possible to formulate the notion of
  {$\F$-separatedness} in terms of coverings, and again, provided filtered colimits
  preserve the class $\F$, the two notions of $\F$-separatedness will agree.
  This is the case for example when $\F$ is the class of $n$-truncated maps.
\end{blanko}

\begin{blanko}{\Qi topoi.}
    Let $\P$ be an $\infty$-topos, and let $L:\P\to\X\subset \P$ be an
    accessible left-exact localization, with right adjoint (inclusion) $G$ and
    unit $\eta$.  Let $(\E,\F)$ be a stable factorization system in $\P$
    with the property that $GL$ preserves
    $\F$, i.e.~if $f$ belongs to $\F$ then also $GLf$ belongs to $\F$.
  We say that an object $Q\in \P$ is {\em $\F$-separated}
  with respect to $L$ if $\eta_Q: Q\to GLQ$ belongs to $\F$.  Let $\Q \subset \P$
  denote the full subcategory consisting of the $\F$-separated objects.  We define
  an {\em \qi topos} to be a presentable $\infty$-category which arises in this way; 
  that is, as the $\infty$-category of
  $\F$-separated objects in an $\infty$-topos equipped
  with an accessible left-exact localization $L$ and a
  stable factorization system $(\E,\F)$ such that $GL$ preserves $\F$.

  An important example of this situation is the factorization
  system ($n$-connected, $n$-truncated): this is stable by 
  Proposition~\ref{nconnstable},
  and any left-exact localization
  preserves $n$-truncated maps by [HTT~5.5.6.16].
\end{blanko}

The following theorem is an $\infty$-version of Proposition~3.3 and Corollary 3.4
of Garner and Lack~\cite{Garner-Lack:1106.5331},
relativized to more general factorization systems.
Our proof draws upon their arguments.
\begin{theorem}\label{thm:PQE}
  Let $\X$ be an $\infty$-topos, presented as a left-exact localization 
  $L:\P\to\X\subset\P$ of an $\infty$-topos $\P$ equipped with a
  stable factorization system 
  $(\E,\F)$ such that $GL$ preserves $\F$,
  and let $\Q\subset\P$ denote the full subcategory of $\F$-separated objects.
  Then the inclusion functors $\overline{G}:\X\to\Q$ and $G':\Q\to\P$ admit
  left adjoints $\overline{L}:\Q\to\X$ and $L':\P\to\Q$, respectively, such
  that $\overline{L}$ is left-exact and $G'L'$ preserves the factorization
  system $(\E,\F)$ and preserves pullbacks over $L'$-local objects (so in particular 
  $L'$ is locally cartesian).
\end{theorem}

\begin{proof}
  Step 1: {\em $G'$ has a left adjoint $L'$.} 
  To establish that $G'$ has a left adjoint, we must show that 
  for every $X\in\P$ the under $\infty$-category $\Q_{X/}$ has an initial 
  object (by the dual of Corollary~\ref{maclanelemma}).
  The left adjoint $L'$ will then be given on objects by sending $X$ to the 
  target of this initial object.
  Consider the $(\E,\F)$-factorization
  of  $\eta: X \to G L X$:
\begin{equation}\label{eta'triangle}
  \xymatrix{
X \ar[rr]^{\eta}\ar[rd]_{\eta'} && G L X \\
&X'  .\ar[ru]_\lambda
}\end{equation}
Since $\lambda$ belongs to $\F$, the object $X'$ is in $\Q$.
Indeed, in the naturality square
\begin{equation}\label{eta'trianglebis}
  \xymatrix{
 G L X\ar[r]^-{\simeq} & GLGLX \\
X' \ar[u]^\lambda \ar[r]_-{\eta_{X'}}&GLX' \ar[u]_{GL\lambda}
}\end{equation}
both $\lambda$ and $GL\lambda$ are in $\F$ (since $GL$ preserves $\F$),
so by the left-cancellation property of right orthogonal classes, also 
$\eta_{X'}$ is in $\F$.
We claim that $\eta':X \to X'$ is an initial object of $\Q_{X/}$.
  Let $f:X\to Y'$ be another object of $\Q_{X/}$.
We must show that the space $\Map_{X/}(\eta',f)$ is contractible.
This space is the fiber over $f$ of the map
  \[
  \Map(X',Y')\too\Map(X,Y')
  \]
  given by precomposition by $\eta'$, so we are done if we can
  show that this map is an equivalence.
Now $\eta':X\to X'$ belongs to $\E$
by construction, and $\eta_{Y'}: Y'\to GLY'$
belongs to $\F$ because $Y'\in \Q$.  This orthogonality relation, 
$\eta' \bot \eta_{Y'}$, is expressed by the 
upper square being a pullback in the diagram
\begin{equation}\label{fact-pb}\xymatrix{
\Map_\P(X', Y')\ar[r]
\ar[d]
\drpullback & \Map_\P(X, Y') 
\ar[d]
\\
 \Map_\P(X', G L Y')\ar[d]_\simeq\ar[r]
& \Map_\P(X, G L Y')\ar[d]^\simeq \\
 \Map_\X(L X', L Y')\ar[r]
& \Map_\X(L X, L Y')
}\end{equation}
where the horizontal maps are pre-composition with $\eta'$ and the
vertical maps of the top square are post-composition with $\eta_{Y'}$.
(The bottom square is just adjunction.)
We now claim that
$L\eta' : LX \to L X'$ is an equivalence.
Indeed, apply $L$ to the triangle \eqref{eta'triangle}
to obtain the following diagram in $\X$:
$$\xymatrix{
LX \ar[rr]^{\simeq}\ar[rd]_{L\eta'} && LG L X \\
&L X'  . \ar[ru]_{L\lambda}
}$$
Note that the class $L\F$ is contained in $(L\E)^\bot$: this is an easy
consequence of the assumption that $GL$ preserves $\F$.  
So in the triangle, $L\lambda$ belongs to 
$(L\E)^\bot$.  It follows\footnote{\label{leftcanc}by the left-cancellation property of 
right orthogonal classes [HTT 5.2.8.6 (3)]} that also 
$L\eta'$ belongs to $(L\E)^\bot$.  But clearly also $L\eta'$ belongs to 
$L\E$,
so altogether it must be an equivalence.
It follows that the bottom map in \eqref{fact-pb} is an equivalence.
Hence the other horizontal maps are equivalences too.

Step 2: 
{\em $G'L'$ preserves both the classes $\E$ and $\F$.}
Indeed, for a map
$f: X \to Y$ in $\P$, consider the diagram
$$
\xymatrix{
X \ar[d]_f \ar[r]^-{\eta'} & G'L'X \ar[d]_{G'L'f} \ar[r]^\lambda & 
GLX \ar[d]^{GLf} \\
Y \ar[r]_-{\eta'} & G'L'Y \ar[r]_\lambda & GLY .
}
$$
If $f$ is in $\E$, then it follows\footnote{\label{rightcanc}by
the right-cancellation property of 
left orthogonal classes [HTT 5.2.8.6 (4)]}
that $G'L'f$ is in $\E$ because so are
the two instances of $\eta'$. On the other hand if $f$ is in $\F$, then
by assumption on $L$, also $GLf$ is in $\F$.  But so are
the two instances of $\lambda$, so we conclude that $G'L'f$ is in 
$\F$ too.\footref{leftcanc}

In the remaining steps, we suppress the inclusion functors $G$ and 
$G'$:
all the arguments take place in $\P$.

Step 3: {\em $L'$ preserves pullbacks along components of $\eta$.}
Consider a pullback diagram in $\P$
$$\xymatrix{
X \!\underset{LX}\times\! A \ar[d] \ar[r] \drpullback & A \ar[d] \\
X \ar[r]_{\eta} & LX .
}$$
Upon applying $L'$ and $L$, we get this commutative diagram in $\P$:
$$
\xymatrix{
X \!\underset{LX}\times\! A \ar[d]_-= \ar[rr]^-{\eta'} 
&& L'\big(X \!\underset{LX}\times\! A\big) \ar[d]_-{\pi'} 
\ar[rr]^-\lambda && L\big(X \!\underset{LX}\times\! A\big) \ar[d]^-\pi \\
X \!\underset{LX}\times\! A \ar[r]_-{\eta'\times\id} 
&L'X \!\underset{LX}\times\! A \ar[r]_-{\id\times\eta'} 
& L'X \!\underset{LX}\times\! L'A \ar[r]_-{\lambda\times\id} 
& LX \!\underset{LX}\times\! L'A\ar[r]_-{\id\times\lambda} 
& LX \!\underset{LX}\times\! LA .
}$$
    Here $\pi$ and $\pi'$ are the canonical comparison maps; we
    must show that $\pi'$ is an equivalence.  From the
    right-hand square we conclude that $\pi'$ is in $\F$:
    indeed, any $\lambda$-map is in $\F$, and so are pullbacks
    of $\lambda$-maps, such as the bottom two maps in that
    square.  But $\pi$ is an equivalence since $L$ is
    left-exact, so the remaining side $\pi'$ must be in $\F$
    too.\footref{leftcanc}  From the left-hand square we conclude that $\pi'$ is
    in $\E$: indeed, $\eta'$-maps belong to $\E$, and by
    basechange stability of $\E$, so do pullbacks of
    $\eta'$-maps, such as the bottom two maps in that square;
    hence the remaining side $\pi'$ must be in $\E$ too.\footref{rightcanc}
    Altogether $\pi'$ is therefore an equivalence, as desired.

Step 4:  {\em
Any map of the form $X \times_{L'X} A \to
X \times_{LX} A$ is in $\F$.}
Indeed, it sits in the pullback diagram
$$\xymatrix{
X \times_{L'X} A \ar[r] \ar[d] \drpullback & L'X \times_{L'X} L'X 
\ar[d]^{\delta_\lambda} \\
X \times_{LX} A \ar[r] & L'X \times_{LX} L'X \,,
}
$$
where the right-hand vertical map is equivalent to the diagonal of
$\lambda:  L'X  \to LX$.  Since $\lambda$ belongs to $\F$, by Lemma~\ref{closedunderdiag}
also $\delta_\lambda$ belongs to $\F$, and hence also its pullback, as claimed.

Step 5:
{\em $L'$ preserves pullbacks over $L'$-local objects.}
By Lemma~\ref{stableunits} below, it is enough to prove that $L'$
preserves pullbacks
along components of $\eta'$ (i.e.~$L'$ has ``stable 
units'').
Applying $L'$ to the pullback diagram 
$$\xymatrix{
P \ar[d]_p \ar[r]^q \drpullback & A \ar[d]^u \\
X \ar[r]_{\eta'} & L'X 
}$$
gives
$$
\xymatrix{
L'P \ar@/^1.6pc/[rr]^{L'q}\ar[rd]_{L'p} \ar[r]^-{a} & L'X \times_{L'L'X} L'A 
\drpullback \ar[d] 
\ar[r]_-\simeq & L'A \ar[d]^{L'u} \\
&L'X \ar[r]_-{L'\eta'}^-\simeq & L'L'X .
}$$
The map $a$ also sits in the diagram
$$\xymatrix{
L'P \ar[r]^-a \ar[d] & L'X \times_{L'L'X} L'A \ar[d] \\
L'(X \times_{LX} A) \ar[r]_\simeq & L'X\times_{LX} L'A  .
}$$
Here the right-hand vertical map is an example of the situation in Step~4, 
so it belongs to $\F$.
The left-hand vertical map is also such an example but with $L'$
applied to it, and since $G'L'$ preserves the class $\F$ by Step~2, 
it is again in $\F$.
The bottom map is an equivalence since $L'$ preserves pullbacks along
components of $\eta$, by Step 3.
It follows from this that also $a$ is in $\F$.
Hence $L'q$ is in $\F$.
On the other hand, since $\eta'$ is in $\E$, by pullback stability of this class,
also $q$ is in $\E$, and since $G'L'$
preserves the class $\E$ by Step 2, also $L'q$ is in $\E$, so
altogether $L'q$ is an equivalence.  Hence $a$ is an equivalence, which is to say
that $L'$ preserves pullbacks along $\eta'$.

Step 6: It is easy to see that $L$ restricts to a functor $\overline L:\Q\to\X$ 
left adjoint to $\overline G:\X\to\Q$.
\end{proof}

\begin{blanko}{Stable units.}
  The condition of preserving pullbacks over local objects is well-known in the literature under the name ``stable units''.
  A localization $L:\P\to\C\subset\P$ is said to have {\em stable 
  units}\footnote{The terminology (in the ordinary-category case) is due to 
  Cassidy--H\'ebert--Kelly~\cite{Cassidy-Hebert-Kelly}.}
  if every pullback of any unit component is inverted by $L$. 
  That is, in any pullback diagram in $\P$ of the form
  $$
  \xymatrix{
  P \drpullback\ar[d]\ar[r]^{f^*(\eta_X)} & A\ar[d]^f \\
  X \ar[r]_{\eta_X} & LX,}$$
  $L(f^*(\eta_X))$ is an equivalence.  Equivalently, $L$
  preserves basechange along unit components.
\end{blanko}

\begin{lemma}\label{stableunits}
  If a localization $L:\P\to\C\subset\P$ has stable units then it preserves
  pullbacks over any object in $\C$ (and conversely).  In particular,
  a localization with stable units is locally cartesian.
\end{lemma}

\begin{proof}
  The statement is that any pullback square in $\P$ of the form
  $$
  \xymatrix{
  P \ar[d]\ar[r]\drpullback & A \ar[d]^-f\\
  B \ar[r]_-g& LX
  }$$
  is preserved by $L$.  But this square decomposes into four pullback squares 
  like this:
  $$
  \xymatrix{
  P \ar[d]\ar[r]\drpullback &\cdot \ar[d]\ar[r]\drpullback & A \ar[d]^{\eta_A}\\
  \cdot \ar[d]\ar[r]\drpullback &\cdot\drpullback \ar[d]\ar[r] & LA \ar[d]^{Lf}\\
  B \ar[r]_{\eta_B}&LB \ar[r]_{Lg}& LX \,.
  }
  $$
  The right-hand bottom pullback square is already inside $\C$ so it is 
  clearly preserved by $L$;
  the other squares are pullbacks along units, so they are also preserved by $L$, by 
  assumption.
\end{proof}

\begin{blanko}{Univalence in \qi topoi.}
  {Theorem~\ref{thm:PQE}} says that any \qi topos arises as an accessible localization
  which preserves pullbacks over local objects and preserves
  the class $\F$.  In the classical case,
  with the (epi,mono) factorization system, this is one of the possible
  characterizations of (Grothendieck) quasitopoi \cite{Garner-Lack:1106.5331}.
  In particular, an \qi topos is presentable and locally cartesian closed. 
  
  The localization $\overline L : \Q \to \X \subset \Q$ now
  satisfies the conditions of Theorem~\ref{thm:induce-along-R}.  Hence:
  
\begin{corollary}
  Let $\Q$ be an \qi topos containing an $\infty$-topos $\X$ as in
  Theorem~\ref{thm:PQE}.  Then any univalent family in $\X$ is also a univalent
  family in $\Q$.
  \qed
\end{corollary}
  In particular,
  since $\X$ is an $\infty$-topos, for each regular cardinal $\kappa$ there is a
  universal (and hence univalent) family classifying all $\kappa$-compact maps, which
  therefore also constitutes a univalent family in $\Q$, although it no longer 
  classifies all maps.  The maps in $\Q$ obtained as pullbacks of
  this univalent family form a bounded local class by 
  Proposition~\ref{prop:univalent=localclass}.
  We finish this section by showing that this bounded local class is the 
  right-orthogonal class 
  of a factorization system.\footnote{following a suggestion by Mike Shulman}
  For this we first need the following general lemma.
\end{blanko}

\begin{lemma}\label{lem:Gq=B}
  Let $L: \Q\to \X\overset G \subset \Q$ be a localization with stable units,
  and let $(\E,\F)$ be a factorization system on $\X$.  Then there is induced
  a factorization system $(\A,\B)$ on $\Q$ in which $\A=L^{-1}\E$ and $\B$
  is the class of pullbacks of maps in $\F$. If $L$ is left-exact, and if
  $(\E,\F)$ is stable, then so is $(\A,\B)$.
\end{lemma}

\begin{proof}
  It is clear that the two classes $\A$ and $\B$ are closed under equivalences.
  For orthogonality, suppose given a square in $\Q$
  $$\xymatrix{
     \cdot \ar[r]\ar[d]_a & \cdot \ar[d]^b \\
     \cdot \ar[r] & \cdot
  }$$
  with $a\in \A$ and $b\in \B$.  The space of fillers for this square is
  equivalent to the space of fillers for the rectangle
  $$\xymatrix{
     \cdot \ar[r]\ar[d]_a & \cdot \ar[d]_b \ar[r] \drpullback & \cdot \ar[d]^{Gf}\\
     \cdot \ar[r] & \cdot \ar[r] & \cdot
  }$$
  expressing $b$ as a pullback of a map from $\F$, and by adjointness this space
  is contractible since $La \bot f$.  Finally we need to establish that every 
  arrow $q$ admits an $(\A,\B)$-factorization.  Just pull back the $(\E,\F)$-factorization
  of $Lq$:
  $$\xymatrix{
     \cdot \ar[rr]^q\ar[dd]_{\eta} \ar[rd]_a&& \cdot \ar[dd]^{\eta} \\
     & \cdot \ar[ru]_b \ar[dd]  \drpullback & \\
     \cdot \ar[rr]^(.27){GLq} \ar[rd]_{Ge}&& \cdot \\
     & \cdot \ar[ru]_{Gf} &
  }$$
  Now by construction, $b\in \B$.  To see that $a\in \A$, apply $L$: since $L$
  has stable units, it preserves pullbacks along units, so the middle vertical
  maps becomes an equivalence, and altogether therefore $La \in \E$.
  
  Finally, if $L$ preserves pullbacks, it is clear that the class $\A$ is stable
  under basechange.
\end{proof}

\begin{corollary}
  {In the situation of Lemma~\ref{lem:Gq=B},} 
  if $p:X\to S$ is a univalent family in $\X$ coming from a local class $\F$
  that happens to be the right orthogonal class of a stable factorization 
  system on $\X$, then the associated univalent family $Gp$ in $\Q$ comes from
  the class $\B$, which for this reason is local.
\end{corollary}

\begin{proof}
  Indeed, the tautological description of the class classified by $Gp$ is
  that it is the class of pullbacks of $Gp$.  Except for the question of bounds,
  this is precisely the description of the class $\B$.
\end{proof}

In our case, $\X$ is the small $\infty$-topos, and $\Q$ is the $\infty$-quasitopos.
We have seen that the general object classifier in $\X$ is still a univalent 
family  in 
$\Q$, and tautologically it classifies maps that can be obtained as pullbacks
of it.  In $\X$ the corresponding local class is the class 
of all maps, $\F= \text{(all)}$, and the corresponding local class in $\Q$ is $\B$,
the class right-orthogonal to the class of local equivalences.
Since the localization $\Q \to \X$ is left-exact, $(\A,\B)$ is a stable
factorization system in this case.  If $\Q$ were an $\infty$-topos, we would know
from Theorem~\ref{M=local} that $\B$ is a local class.  In the present case,
$\Q$ is not an $\infty$-topos, but the fact that the factorization system is
induced from an $\infty$-topos makes the conclusion valid anyway: since the 
right-orthogonal class $\B$ is the class of pullbacks of a univalent family, we know it is local.

The situation  described also covers the classical case: in the classical case, $\Q$
is a quasitopos and $\X$ is the subtopos of {\em coarse objects}, {i.e.~the 
objects right orthogonal to all morphisms that are simultaneously epi and 
mono}.
In this case, the strong-subobject classifier in $\Q$
is actually coarse, hence is induced from $\X$ as in the general situation we 
have
described.  Equivalently, in terms of factorization systems, the classical
localization $L : \Q\to \X \subset \Q$ preserves and detects epis; hence to
the class right-orthogonal to the epis in $\X$ corresponds the class right-orthogonal
to the epis in $\Q$.  These are the class of all monos in the topos $\X$ and the 
class of strong monos in the quasitopos $\Q$.

\section{Bundles and connected univalent families}
\label{sec:bundles}

In this section, we study the univalent family $p:X\to S$ associated to a given object $F$ of an $\infty$-topos $\X$.
Specifically, we show that the universal $F$-bundle over $B\underline{\Aut}(F)$ is univalent; these provide examples of univalent families with connected base (unlike the universal univalent families, which have many connected components).
Finally, we consider some sporadic examples, outside the realm of $\infty$-topoi.
\begin{blanko}{Bundles and automorphisms.}
Recall that an $F$-bundle (or a bundle with fiber $F$) in an $\infty$-topos $\X$
is a map $X\to S$ for which 
there exist
an effective epi $T \to S$ such that the pullback $Y=X\times_S T \to T$ is equivalent (as an object over $T$) to the trivial $F$-bundle, the projection $F \times T \to T$.
Also recall that $F$ defines a sheaf $\underline{\Aut}(F)$ of automorphisms of $F$, whose $T$-points is the space $\Aut_{/T}(F\times T)$ of automorphisms of $F\times T$ over $T$.
Finally, if $X$ is an object of $\X$ equipped with a (right) action $G\to\underline{\Aut}(X)$ by a group object $G$ of $\X$, we write $X/G$ for the quotient of this action, given by the geometric realization of the simplicial object which in degree $n$ is $X\times G^n$.
\end{blanko}

\begin{lemma}\label{Fbundles}
  For $F$ an object of an $\infty$-topos,
  the class of $F$-bundles is a bounded local class,
  and the corresponding univalent family (the universal  
  $F$-bundle)
  is $$F/\underline{\Aut}(F) \too */\underline{\Aut}(F) =
B\underline{\Aut}(F).$$
\end{lemma}

\begin{proof}
  We apply the criterion of \cite[Proposition~6.2.3.14]{Lurie:HTT}.
  It is clear that $F$-bundles are stable under pullbacks and sums, 
  so it remains to check that
  if 
  $$\xymatrix{
  X' \drpullback \ar[r] \ar[d]_{f'} & X \ar[d]^f\\
  Y' \ar[r]_e & Y}
  $$
  is a pullback along an effective epi $e$,
  and if $f'$ is an $F$-bundle, then also $f$ is an $F$-bundle.
  But
  to say that $f'$ is an $F$-bundle 
  means that there exists an effective epi $Y''\to Y'$ that
  trivializes it.  That is to say that the pullback of $f'$ to $Y''$
  is a trivial $F$-bundle. But then the composite effective epi $Y'' \to Y' 
  \to Y$ also trivializes $f$, which is therefore an $F$-bundle.
  
  To check that $F/\underline{\Aut}(F) \too */\underline{\Aut}(F)$ is 
  the universal family for $F$-bundles, it suffices to show that the 
  sheaf $\B_F$ which associates to $S$ the space of $F$-bundles $X\to S$
  is represented by $B\underline{\Aut}(F)$.
  Since any $F$-bundle is locally trivial, we conclude that there exists
  an effective epi $*\to\B_F$.
  Moreover, the pullback $*\times_{\B_F} *$ is equivalent to the group 
  object $\underline{\Aut}(F)$ of automorphisms of the trivial $F$-bundle $F\to *$.
  Hence $\B_F$ is equivalent to the quotient of the action of 
  $\underline{\Aut}(F)$ on $*$, the definition of $B\underline{\Aut}(F)$.
\end{proof}

\begin{blanko}{Sections and pointed objects.}
  Typically, given a family $p:X\to S$ and a map $f:T\to S$, the pullbacks
  $f^*p:f^*X\to T$ do not admit sections.
They do, however, if $f$ happens to factor through $p$ itself.
In particular, the pullback $p^*p:p^*X\to X$, which is none other 
than the projection $\pi:X\times_S X\to X$, admits a diagonal section
$\delta:X\to X\times_S X$.
It follows that we have a commutative square
\begin{equation}\label{square}
\xymatrix{
\Map(T,X)\ar[r]\ar[d] & \C_{/T\ast}^{\eq}\ar[d]\\
\Map(T,S)\ar[r] & \C_{/T}^{\eq}}
\end{equation}
in which the upper right-hand corner is the (maximal subgroupoid of) 
the $\infty$-category of objects over 
and under $T$, i.e.~the $\infty$-category $(\C_{/T})^{\eq}_{\id_{T}/}$ 
of pointed objects of $\C_{/T}$, and the top horizontal map
sends $f:T\to S$ to the pullback of $X\stackrel\delta\to X\times_S X 
\stackrel \pi\to X$ along $f$.
\end{blanko}

\begin{proposition}\label{section}
Let $\C$ be a locally cartesian closed $\infty$-category and let $p:X\to S$ 
be a univalent family in $\C$.
Then, for any object $T$ of $\C$, the map
\[
\Map(T,X)\too\C_{/T*}^{\eq}
\]
described above
is a monomorphism.
\end{proposition}

\begin{proof}
Choose a basepoint $f:T\to S$ of $\Map(T,S)$ and write $f^*X$ 
for its image in $\C_{/T}^{\eq}$.
We have the commutative diagram
$$\xymatrix{
\Map_{/S}(f,p) \ar[d]\ar[rr]^\simeq && \Map_{/T}(\id_{T},f^*p) \ar[d] \\
\Map(T,X) \ar[d]\ar[r] &\C_{/T}^{\eq}{}_{*} \drpullback \ar[d]\ar[r]& \C_{/T}{}_{*}  \ar[d] \\
\Map(T,S) \ar[r] &\C_{/T}^{\eq} \ar[r]& \C_{/T}    .
}$$
The outer vertical maps form fiber sequences (over $f$ and $f^*p$ 
respectively).  
The right-hand bottom square is a pullback square since its 
right-hand side is conservative.
The top vertical map of the diagram is an equivalence by adjunction,
and since this holds for any object $f \in \Map(T,S)$, this implies
that the bottom outer square is a pullback.  Hence the left-hand 
bottom square is a pullback too.  Finally since $p$ is univalent, the
map $\Map(T,S) \to\C_{/T}^{\eq}$ is mono, and hence the map
$\Map(T,X)\too\C_{/T*}^{\eq}$ is mono too, as asserted.
\end{proof}

\begin{blanko}{$n$-Truncated univalent families in $(n+1)$-topoi.}\label{BG}
  It is tempting to suppose that any univalent family in an $(n+1)$-topos must
  belong to the class of $(n-1)$-truncated morphisms.  Indeed, an
  $(n+1)$-category is an $(n+1)$-topos if and only if it is presentable, locally
  cartesian closed, and the $(n-1)$-truncated maps form a local
  class.{\footnote{[HTT, Theorem 6.4.1.5]}} While this
  {supposition} is true for $n=\infty$, {it} fails in general for
  $n<\infty$.

Consider, for instance, the case $\C=\tau_{\leq n}\S$, the $(n+1)$-topos of $n$-truncated spaces.
By \ref{Fbundles}, any $n$-truncated space $F$ potentially gives rise to a univalent family
\[
p:F/\Aut(F)\too */\Aut(F)\simeq B\Aut(F)
\]
classifying morphisms with fiber $F$; however, $B\Aut(F)$ is typically not $n$-truncated, 
so this map need not live in $\C$. (Of course, if $F$ is $(n-1)$-truncated, 
then $\Aut(F)$ is also $(n-1)$-truncated and consequently $B\Aut(F)$ is $n$-truncated.)
Hence counterexamples arise when $F$ is only $n$-truncated but $\Aut(F)$ is $(n-1)$-truncated.

In particular, if $G$ is a discrete group and $F=BG$, which we regard as a pointed 
connected space, then there is a fiber sequence
\[
\Aut_*(BG)\too\Aut(BG)\too BG
\]
in which, by the equivalence between pointed connected objects and 
group objects of an $\infty$-topos,\footnote{[HTT 7.2.2.11].} we may identify the 
fiber $\Aut_*(BG)$ with the space of 
group automorphisms of $G$, which is discrete since $G$ is discrete.
Using the long exact sequence
\[
*\cong\pi_1\Aut_*(BG)\too\pi_1\Aut(BG)\too\pi_1 BG\cong G
\]
we calculate that $\pi_1\Aut(BG)$ is isomorphic to the kernel of the 
conjugation action map
\[
G\too\Aut_\mathrm{Gp}(G).
\]
Hence $\pi_1\Aut(BG)\cong Z(G)$, the center of $G$, and we conclude 
that $\Aut(BG)$ is discrete if and only if $G$ has trivial center.
In this case, $\Aut(BG)\simeq\pi_0\Aut(BG)\cong\mathrm{Out}(G)$, and 
we obtain a univalent family
\[
p:BG/\mathrm{Out}(G)\too B\mathrm{Out}(G)
\]
in $\mathrm{Gpd}\simeq\tau_{\leq 1}\S$ which is not $0$-truncated.
Note that there are infinite families of groups with trivial center, 
e.g.~the dihedral groups $D_n$ for $n$ odd, the symmetric groups 
$\Sigma_n$ for $n>2$, or any simple group.
\end{blanko}

\begin{blanko}{Some univalent families outside topoi.}
An interesting example of a presentable locally cartesian closed $\infty$-category 
$\C$ which is not an $n$-topos for any $0\leq n\leq\infty$ is the 
$\AA^1$-localization of the $\infty$-topos $\mathrm{Shv_{Nis}(Sm}_S)$ of 
Nisnevich sheaves of spaces on smooth $S$-schemes of 
Morel--Voevodsky~\cite{Morel-Voevodsky:IHES}.
Here $S$ denotes a fixed base scheme, which we take to be integral, Noetherian, 
and of finite Krull dimension (e.g.~$S=\mathrm{Spec}(k)$ for $k$ a field).
Recall that a Nisnevich sheaf of groups $G$ on $\mathrm{Sm}_S$ is said to be 
{\em strongly $\AA^1$-invariant} if, in the $\infty$-topos 
$\mathrm{Shv_{Nis}(Sm}_S)$, its classifying space $BG$ is $\AA^1$-local.
\end{blanko}

\begin{proposition}\cite[Remark 3.5]{SO}
The $\AA^1$-localization $\C=\mathrm{Shv_{Nis}^{\AA^1}(Sm}_S)$ of 
$\mathrm{Shv_{Nis}(Sm}_S)$ is a presentable locally cartesian 
$\infty$-category which is not an $n$-topos for any $0\leq n\leq\infty$.
\end{proposition}

\begin{proof}
Given a map of sheaves $f:U\to T$ and 
a smooth $S$-scheme $X\to T$ over $T$, we have
\[
f^*(X\times\AA^1)\cong (f^*X)\times\AA^1,
\]
which is to say that basechange preserves the generating $\AA^1$-local equivalences.
Hence the localization is \loccart\ by Proposition~\ref{loccartstars}.
Clearly $\C$ is not an $n$-topos for any $n<\infty$ since it contains objects which are not $(n-1)$-truncated, so we must show that $\C$ is not an $\infty$-topos.
This follows from the fact that there exist group objects $G$ of $\C$ which are not strongly $\AA^1$-local, so that the natural map $G\to\Omega L^{\AA^1}BG$ cannot be an equivalence.
\end{proof}

\begin{blanko}{Univalent families in motivic homotopy theory.}
Let $F$ be an $\AA^1$-local object of $\mathrm{Shv_{Nis}(Sm}_S)$ such that 
$G=\underline{\Aut}(F)$ is a strongly $\AA^1$-invariant group object.
Then $p:F/G\to */G\simeq BG$ is a univalent family in 
$\mathrm{Shv_{Nis}(Sm}_S)$ such that the source and target are $\AA^1$-local
objects, since $F/G$ sits in a fibration $F\too F/G\too BG$ and is therefore 
also $\AA^1$-local.
\end{blanko}

\begin{lemma}
  Let $G$ be a strongly $\AA^1$-invariant group over $S$ with the property
  that $\underline{\Aut}_*(G)$ is also strongly $\AA^1$-invariant.
  Then $G/\underline{\Aut}(G)\to B\underline{\Aut}(G)$ is a univalent 
  family in $\mathrm{Shv_{Nis}^{\AA^1}(Sm}_S)$.
  \qed
\end{lemma}

\begin{example}\label{Gm}
  Take for instance $F=\GG_m$, the multiplicative group over $S$.
  Since we have assumed $S$ integral, there are no nonconstant units 
  of $\AA^1$, so that $\GG_m$ is $\AA^1$-local.
  To calculate $\underline{\Aut}(\GG_m)$, we use the fact that
  \[
  \GG_m[2]\cong\underline{\Aut}_*(\GG_m)\too\underline{\Aut}(\GG_m)\too\GG_m
  \]
  is a fibration sequence, so it deloops to a fibration sequence
  \[
  B\GG_m[2]\too B\underline{\Aut}(\GG_m)\too B\GG_m
  \]
  (here $\GG_m[2]$ denotes the $2$-torsion subgroup of $\GG_m$, which is 
  a finite discrete group scheme over $S$).
  Since both $B\GG_m$ and $B\GG_m[2]$ are $\AA^1$-local, we obtain a
  univalent family by the lemma.
\end{example}

\begin{example}\label{E}
  One obtains a similar result for elliptic curves $C$, 
  which are $\AA^1$-local since any map $\AA^1_T\cong\AA^1\times_S T\to C$ 
  over $S$ induces a map $\AA^1_T\to C_T$ over $T$ and thus can be 
  completed to a map $\PP^1_T\to C_T$ over $T$.
  But $C_T$ has no rational curves, so any such map must be constant.
  Moreover, $\underline{\Aut}_*(C)$ is a finite group $S$-scheme, 
  as the moduli stack $\M_{1,1}$ of genus one curves with one marked 
  point is Deligne--Mumford.
  It follows that $\underline{\Aut}_*(C)$ is also strongly $\AA^1$-invariant.
\end{example}

\section{Univalence in combinatorial type-theoretic model categories}
\label{sec:type}

The univalence property is a homotopy invariant notion and is therefore
independent of the particular features of a model.  However, in order to get a
literal interpretation of type theory in homotopy theory, certain strictness
features seem to be required~\cite{Shulman:1203.3253}; these are not homotopy
invariant and are therefore features manifest only on the level of the model
category.  In this section we show that the standard Quillen model category $\M$
associated to a presentable locally cartesian closed $\infty$-category $\C$ is a
``combinatorial type-theoretic model category,'' and that any univalent family in $\C$ lifts
to a univalent fibration in $\M$.

\begin{blanko}{Combinatorial type-theoretic model categories.}\label{def:typemodel}
  By a {\em combinatorial type-theoretic model category} we understand a proper combinatorial
  model category in which the cofibrations are exactly the monomorphisms, and
  whose underlying category is locally cartesian closed.
    These conditions, which are natural from the viewpoint of
    $\infty$-categories, are slightly stronger than the five axioms
    proposed by Shulman~\cite{Shulman:1203.3253} (which in turn are  
    more restrictive than the notion of logical model category of Arndt--Kapulkin~\cite{Arndt-Kapulkin:1208.5683}), but if we add
    ``combinatorial'' to his axioms, then the
    two notions are equivalent in the sense that    
    any $\infty$-category presented by a model category of one sort
is also presented by one of the other; this is a
    consequence of Theorem~\ref{lcc=>type} below.
  The relationship between proper model categories and locally cartesian closed
  $\infty$-categories has been considered recently in~\cite{Cisinski:blogpost}, 
  where 
  Cisinski  outlined a proof of Theorem~\ref{lcc=>type} below.  Our proof is somewhat 
  different from his.
\end{blanko}

\begin{lemma}\label{rightproper}
  Let $\M$ be a model category in which the cofibrations are the monos.
  If, for any fibration $f:T\to S$ between fibrant objects, the pullback
  functor $f^* : \M_{/S} \to \M_{/T}$ preserves trivial cofibrations,
  then $\M$ is right proper.
\end{lemma}

\begin{proof}\hspace*{-4pt}\footnote{Cisinski~\cite{Cisinski:blogpost} derives this fact (in 
  the case  where $\M$ is a topos) from a fancier result concerning
  localizers~\cite[Thm.~4.8]{Cisinski:THT}.  We think the present elementary
  proof deserves to be known.}
  Let $g: V \to W$ be any fibration, and let $w:W' \to W$ be a weak equivalence;
  we need to show that in the pullback square
  $$\xymatrix{
     V' \drpullback \ar[r]\ar[d]_v & W' \ar[d]^w \\
     V \ar[r]_g & W
  }$$
  the map $v$ is again a weak equivalence.
  Let $f: T \to S$ be a
  fibrant replacement of $g$, and let $\tilde f$ denote
  the pullback of $f$ to $W$, as in the diagram
$$
\xymatrix{
V \ar[r]_-{j}\ar@/_0.4pc/[rd]_{q_V}\ar@/^1.6pc/[rr]^g & T\times_S W 
\drpullback\ar[r]_-{\tilde f}\ar[d]_m & W \ar[d]^{q_W} \\
& T \ar[r]_f & S   .
}$$
Here $q_V$ and $q_W$ are trivial cofibrations,
and $f:T\to S$ is a fibration between
fibrant objects. Since $f^*$ preserves trivial cofibrations,
also $m$ is a trivial cofibration.
Since $q_W$ and $m$ are monos, $j$ is again a mono.
Since $q_V$ and $m$ are weak equivalences, $j$ is again a weak equivalence,
so $j$ is a trivial cofibration.  In the top triangle, written as
$$
\xymatrix{
V \ar[d]_j \ar@{=}[r] & V \ar[d]^g \\
 T\times_S W  \ar@{..>}[ru] \ar[r]_-{\tilde f} & W  ,
}$$
since $j$ is a trivial cofibration and $g$ is a fibration, there is a diagonal 
filler, so in particular
$j$ has a retraction.
Now pull back this whole filler diagram along
the weak equivalence $w: W' \to W$:
$$\xymatrix @!0 @R=30pt @C=60pt {
V' \ar[dd]_v \ar@{=}[rr] \ar[rd]_i && V' \ar[dd]_(0.7)v \ar[rd] \\
& T\times_S W' \ar[dd]_(0.28)u \ar[rr] \ar@{..>}[ru] && W' \ar[dd]^w \\
V \ar@{=}[rr] \ar[rd]_-{j} && V \ar[rd]^g & \\
& T\times_S W \ar@{..>}[ru] \ar[rr]_-{\tilde f} && W
}$$
Here $u$ is an equivalence, because it is the pullback of $w$ along $f$,
and $f^*$ preserves trivial cofibrations by assumption, and trivial
fibrations, as does any pullback functor.  Finally, since $v$ is a retract of 
$u$, it is a weak equivalence too, as required.
\end{proof}

\begin{blanko}{Underlying $\infty$-categories.}\label{underlying}
We now consider the underlying $\infty$-category of a combinatorial type-theoretic model category $\M$, which we define as the localization
\[
\C=\mathrm{N}(\M)[W_\M^{-1}]
\]
of the nerve of $\M$ along the weak equivalences $W_\M$ of $\M$ \cite[Definition
1.3.4.15]{Lurie:HA} (note that since $\M$ is combinatorial type-theoretic, monomorphisms are
cofibrations, so every object of $\M$ is automatically
cofibrant).\footnote{Provided a model category $\M$ has functorial
factorizations, which is almost always satisfied in practice and automatic when
$\M$ is combinatorial, there is no need to restrict to the cofibrant objects
when forming the underlying $\infty$-category $\mathrm{N}(\M)[W_\M^{-1}]$; see
\cite[Remark 1.3.4.16]{Lurie:HA}.} The formation of the underlying $\infty$-category is
evidently functorial in left Quillen functors of combinatorial type-theoretic model
categories, as left Quillen functors preserve weak equivalences of cofibrant
objects (and all objects are cofibrant).

\begin{lemma}\label{lem:l1}
Let $\M$ be a combinatorial type-theoretic model category and let $S$ be a fibrant object of $\M$.
Then the canonical functor $\mathrm{N}(\M_{/S})[W_{/S}^{-1}]\to\C_{/S}$ is an equivalence.
\end{lemma}

\begin{proof}
Since the functor is clearly essentially surjective, it is enough to show that it is also fully faithful.
This follows from the fact that the derived mapping space in the slice $\M_{/S}$, between the two objects $X\to S$ and $Y\to S$, is given as the homotopy fiber
\[
\Map_{/S}(X,Y)\too\Map(X,Y)\too\Map(X,S)
\]
over the point $X\to S$ in $\Map(X,S)$ (see \cite{DS} for an in-depth comparison of derived mapping spaces in model categories and mapping spaces in quasicategories).
\end{proof}

Since $\M$ is in particular a combinatorial model category, we may form the Bousfield localization of $\M$ with respect to any set of morphisms $V$ of $\M$; this will be denoted $V^{-1}\M$.
On the other hand, we may form the Bousfield localization the underlying $\infty$-category $\C$ of $\M$, which will similarly be denoted $V^{-1}\C$.
Note that $V^{-1}\C$ satisfies the following universal property: for any presentable $\infty$-category $\D$,
\[
\Fun^\mathrm{L}(V^{-1}\C,\D)\subset\Fun^\mathrm{L}(\C,\D)
\]
is the full subcategory the $\infty$-category $\Fun^\mathrm{L}(\C,\D)$ of
colimit-preserving functors from $\C$ to $\D$ which send the arrows in $V$ to
equivalences \cite[Proposition 5.5.4.20]{Lurie:HTT}.  From this we deduce
that Bousfield localization is
compatible with passage to the underlying $\infty$-category; this is the content of
the following proposition.

\begin{lemma}\label{prop:BL}
The functor $V^{-1}\C\to\mathrm{N}(V^{-1}\M)[W_{V^{-1}\M}^{-1}]$, induced from the colimit-preserving functor of presentable $\infty$-categories $\C=\mathrm{N}(\M)[W_\M^{-1}]\to\mathrm{N}(V^{-1}\M)[W_{V^{-1}\M}^{-1}]$ which inverts the arrows in $V$, is an equivalence.
\end{lemma}

\begin{proof}
As a category, $V^{-1}\M$ is equal to $\M$; they only differ as model categories.
Hence we have a composite functor $\mathrm{N}(V^{-1}\M)=\mathrm{N}(\M)\to\C\to V^{-1}\C$, and this functor sends the arrows of $W_{V^{-1}\M}$ (the $V$-local weak equivalences in $\M$) to equivalences in $V^{-1}\C$, essentially by definition of local weak equivalence.
We therefore obtain a functor $\mathrm{N}(V^{-1}\M)[W_{V^{-1}\M}^{-1}]\to V^{-1}\C$ which is inverse (up to equivalence) to the functor $V^{-1}\C\to\mathrm{N}(V^{-1}\M)[W_{V^{-1}\M}^{-1}]$.
\end{proof}

Since we are primarily concerned with locally cartesian closed $\infty$-categories, 
it is natural to ask whether 
passing to the slice over an object of a combinatorial type-theoretic model category is 
compatible with passing to the slice on the level of underlying $\infty$-categories.

\begin{lemma}\label{lem:last}
With $V$ as above, for each $V$-local fibrant object $S\in \M$, let $V_S$ denote the inverse image of $V$ under the projection $\M_{/S}\to\M$.
Then the induced Bousfield localization $\M_{/S} \to V_S^{-1} \M_{/S}$ has as underlying 
functor of $\infty$-categories $\C_{/S} \to V_S^{-1} \C_{/S}$.
Moreover, there is a canonical equivalence $V_S^{-1}\C_{/S}\simeq (V^{-1} \C)_{/S}$.
\end{lemma}

\begin{proof}
The first claim is an immediate consequence of the previous lemma.
We have a natural map $V^{-1}_S\C_{/S}\too (V^{-1}\C)_{/S}$
by the universal property of Bousfield localization: the arrows in $V^{-1}_S$ become equivalences in $(V^{-1}\C)_{/S}$.
We claim that an object $X\to S$ of $\C_{/S}$ is $V_S$-local if and only if $X$ is $V$-local.
The ``if'' direction is clear; for the ``only if'' direction, let $f:U\to T$ be an arrow of $V$.
This yields a commutative square
\[
\xymatrix{
\Map(T,X)\ar[r]\ar[d] & \Map(T,S)\ar[d]\\
\Map(U,X)\ar[r] & \Map(U,X)}
\]
in which the right vertical map is an equivalence.
It is therefore enough to check that, for each point $p:T\to S$ in $\Map(T,S)$, the fibers (over $p$ and $q=p\circ f$) are equivalent.
But these fibers are $\Map_S(T,X)$ and $\Map_S(U,X)$, respectively, and they are equivalent since $X$ is $V_S$-local by assumption.

Thus we obtain a commutative triangle
\[
\xymatrix{V^{-1}_S\C_{/S}\ar[rr]\ar[rd] & & (V^{-1}\C)_{/S}\ar[ld]\\
& \C_{/S} &}
\]
in which both functors into $\C_{/S}$ are fully faithful.
It follows that the top map is fully faithful, and it is easily seen to be essentially surjective as well.
\end{proof}
\end{blanko}

\begin{proposition}\label{wtype=>lcc}
  Let $\M$ be a combinatorial type-theoretic model category.  Then the underlying
  $\infty$-category $\C=\mathrm{N}(\M)[W^{-1}]$ of $\M$ is presentable and locally
  cartesian closed.
\end{proposition}

\begin{proof}
  Since $\M$ is combinatorial, $\C$ is presentable.\footnote{[HA 1.3.4.22]} Let
  $f:T\to S$ be an arrow of $\C$, which (replacing $f$ by an equivalent arrow if
  necessary) we may assume is a fibration between fibrant objects in $\M$.  Then
  we have equivalences
\[
\C_{/S}\simeq\mathrm{N}(\M_{/S})[W_S^{-1}],
\]
and similarly for $T$.
Since by hypothesis, $f^*:\M_{/S}\to\M_{/T}$ is a left Quillen functor, 
it follows\footnote{[HA 1.3.4.26]}
that the induced functor
\[
f^*:\C_{/S}\simeq\mathrm{N}(\M_{/S})[W_S^{-1}]\too\mathrm{N}(\M_{/T})[W_T^{-1}]\simeq\C_{/T}
\]
admits a right adjoint $f_*$.
\end{proof}

The following result can be regarded as a converse result, stated in a relative
setting.

\begin{proposition}\label{lccloc=>type}
Let $\M$ be a combinatorial type-theoretic model category, let $V$ be a class of maps of $\M$ of small 
generation\footnote{\cite[Remark 5.5.4.7]{Lurie:HTT}}, 
and let $\M\to V^{-1}\M$ be the left Bousfield localization of $\M$ along $V$.
If the induced localization of $\C= \mathrm{N}(\M)[W^{-1}]$,
\[
L:\C\too V^{-1}\C,
\]
is locally cartesian then $V^{-1}\M$ is again a combinatorial type-theoretic model category.
\end{proposition}

\begin{proof}
To show that $V^{-1}\M$ is combinatorial type-theoretic, the only non-trivial point is
to establish that it is right proper.  By Lemma~\ref{rightproper}, 
it is enough to show that
$f^*$ preserves local equivalences whenever $f:T\to S$ is a
$V$-local fibration between $V$-local fibrant objects.
Since $S$ is a $V$-local fibrant object, we have equivalences
\[
\C_{/S}\simeq\mathrm{N}(\M_{/S})[W_S^{-1}]
\]
by Lemma \ref{lem:l1} and
\[
V_S^{-1}\C_{/S}\simeq\mathrm{N}(V_S^{-1}\M_{/S})[W_S^{-1}]
\]
by Lemmas \ref{prop:BL} and \ref{lem:last}, as well as the corresponding equivalences involving $T$.
There is an induced pullback functor
\[
f^*:V_S^{-1}\C_{/S}\simeq (V^{-1}\C)_{/S}\too (V^{-1}\C)_{/T}\simeq 
V_T^{-1}\C_{/T} .
\]
To say that the localization $\C \to V^{-1}\C$ is
\loccart\
means that the following diagram commutes:
\[
\xymatrix{
\C_{/S}\ar[r]^-{L_S}\ar[d]_{f^*} & V_S^{-1}\C_{/S}\ar[d]^{f^*}\\
\C_{/T}\ar[r]_-{L_T} & V_T^{-1}\C_{/T}.
}
\]
Hence
\[
L_T f^*:\C_{/S}\to V_T^{-1}\C_{/T}
\]
inverts maps in $V_S$ (i.e.~carries every morphism in $V_S$ to an equivalence).
It follows that, on the level of model categories,
\[
f^*:\M_{/S}\too\M_{/T}
\]
sends $V_S$-local equivalences to $V_T$-local equivalences,
as required.
\end{proof}

\begin{blanko}{Constructing models.}
  Given a presentable $\infty$-category $\C$, one can use its standard presentation as a
  localization of a presheaf $\infty$-category to construct a combinatorial
  simplicial model category $\M$ whose underlying $\infty$-category is equivalent to
  $\C$.  This model category can be obtained as a left Bousfield localization of
  a simplicial presheaf (model) category, essentially by copying the
  generators and relations.  While this construction is not unique, it is
  sort of standard, and we shall refer only to this construction when talking
  about a model of $\C$.
  We recall the construction\footnote{\cite[Proposition~A.3.7.6]{Lurie:HTT}}
  as we will require some details of it.
  
  Since $\C$ is presentable, it is by definition an accessible
  localization (for some sufficiently large cardinal $\kappa$) of the
  $\infty$-category $\mathrm{Pre}(\C^\kappa)$ of presheaves (of $\infty$-groupoids) on
  the $\infty$-category $\C^\kappa$ of $\kappa$-compact objects in $\C$.  More
  specifically, writing $U$ for the (small) set of maps of the form
\[
\colim_i \Map(-,x_i)\too\Map(-,\colim_i x_i),
\]
for $x:I\to\C^\kappa$ a diagram indexed on a $\kappa$-small simplicial set $I$,
we have an equivalence $\C\simeq U^{-1}\mathrm{Pre}(\C^\kappa)$.
Note that $\mathrm{Pre}(\C^\kappa) \to \C$ is \loccart\ by Proposition~\ref{lcc=lcc}.

Let $\simplC$ be a fibrant simplicial category such 
that $\C^{\kappa}\simeq\mathrm{N}(\simplC)$,
and consider the simplicial category $\Pre_\Delta(\simplC)$ of 
simplicial presheaves on $\simplC$,
endowed with the injective model structure.
Let $V$ be the smallest saturated class generated by the set of maps of the form
\[
\hocolim_i\Map(-,x_i)\too\Map(-,\hocolim_i x_i)
\]
in $\Pre_\Delta(\simplC)$ corresponding to the maps in $U$,
and let $\M$ be the left Bousfield localization of $\Pre_\Delta(\simplC)$ by $V$.
This $\M$ is a model for $\C$,
in the sense that
we have equivalences
\[
\C\simeq U^{-1}\Pre(\C^\kappa)\simeq 
V^{-1}\mathrm{N}(\Pre_\Delta(\simplC))[W^{-1}],
\]
where $W$ denotes the weak equivalences in $\Pre_\Delta(\simplC)$.
By Lemma \ref{prop:BL}, it is clear that the induced localization on the level of $\infty$-categories is precisely $\mathrm{Pre}(\C^\kappa)\to \C$.
\end{blanko}

Our main theorem in this section is the following.

\begin{theorem}\label{lcc=>type}
Let $\C$ be a presentable locally cartesian closed $\infty$-category.
Let $\M$  be the combinatorial simplicial model category constructed above,
whose underlying $\infty$-category is equivalent to $\C$.
Then $\M$ is a combinatorial type-theoretic model category.
\end{theorem}

\begin{proof}
  We shall apply Proposition~\ref{lccloc=>type} to the left Bousfield localization
  $\Pre_\Delta(\simplC) \to \M$, so we must first establish that $\Pre_\Delta(\simplC)$ 
  is itself a combinatorial type-theoretic model category.  First of all,
  $\Pre_\Delta(\simplC)$ is locally cartesian closed
  as an ordinary category: the simplicial category
  $\simplC$ can be viewed as an internal category object in $\Set_\Delta$, and
  accordingly $\Pre_\Delta(\simplC)$ can be viewed as the category of
  presheaves on $\simplC$ internally to $\Set_\Delta$.  But this is always a 
  topos.\footnote{See for example Theorem 6.5 of Barr--Wells~\cite{Barr-Wells}.}
  Second, $\Pre_\Delta(\simplC)$ is proper and has the monos as cofibrations
  since the model category of simplicial sets (with its standard ``Kan'' model
  structure) satisfies these conditions.  So $\Pre_\Delta(\simplC)$ is a
  combinatorial type-theoretic model category.  Now the left Bousfield localization
  $\Pre_\Delta(\simplC) \to \M$ induces the standard presentation
  $\Pre(\C^\kappa) \to \C$ at the $\infty$-category level, and since this
  localization is \loccart\ by Proposition~\ref{lcc=lcc}, we conclude by
  Proposition~\ref{lccloc=>type}  that $\M$ is a combinatorial type-theoretic model category.
\end{proof}

\begin{definition}
  Let $\M$ be a combinatorial type-theoretic model category and let $p:X\to S$ be a 
  fibration between fibrant objects of $\M$.
  Then $p$ is said to be {\em univalent} if the map
  \[
  u:\Map_{/S\times S}(-,S)\too\sEq_{/S}(X,X).
  \]
  is an isomorphism of $\H$-enriched presheaves.  
  Here $\H$ denotes the homotopy category of simplicial sets.
\end{definition}

\begin{proposition}\label{prop:univ=univ}
  Let $\M$ be a combinatorial type-theoretic model category.  A fibration $p:X\to S$ between
  fibrant objects of $\M$ is univalent in the model category $\M$ if and only if
  it is univalent in the underlying $\infty$-category $\mathrm{N}(\M)[W^{-1}]$.
\end{proposition}

\begin{proof}
  In general, if $\C\to\D$ is a functor of $\infty$-categories, then it induces a
  functor $\Ho(\C)\to\Ho(\D)$ of $\H$-enriched categories.  Moreover, this
  assignment is natural, in that it induces a functor
  \[
  \Ho(\Fun(\C,\D))\too\Fun_\H(\Ho(\C),\Ho(\D)),
  \]
  where the latter denotes the category of $\H$-enriched functors, which 
  is conservative since a (homotopy class of) natural transformation of 
  functors from $\C$ to $\D$ is a natural equivalence if an only if the 
  resulting natural transformation of $\H$-enriched functors $\Ho(\C)\to\Ho(\D)$ 
  is a natural isomorphism.

  In particular, taking $\C=\mathrm{N}(\M_{/S\times S})[W^{-1}]^{\op}$ 
  and $\D=\mathrm{N}(\Set_\Delta)[W^{-1}]$, we obtain a conservative functor
  \[
  \Ho(\Pre(\mathrm{N}(\M_{/S\times S})[W^{-1}]))\too\Pre_{\H}(\Ho(\M_{/S\times S})).
  \]
  By definition, $p:X\to S$ is univalent in $\M$ if and only if $u$ is a 
  natural isomorphism in $\Pre_{\H}(\Ho(\M_{/S\times S}))$, if and only if 
  $u$ is a natural equivalence in $\Ho(\Pre(\mathrm{N}(\M_{/S\times S})[W^{-1}]))$.
\end{proof}

\newcommand{\arxiv}[1]{ArXiv:#1}


\small

\begin{thebibliography}{8}

\bibitem{Arndt-Kapulkin:1208.5683}
P.~Arndt and K.~Kapulkin,
{\em Homotopy-theoretic models of type theory},
in: {\em Typed lambda calculi and applications},
Lecture Notes in Comput. Sci. {\bf 6690} (2011),
45--60,
Springer--Verlag, Heidelberg.
\arxiv{1208.5683}.

\bibitem{oberwolfach2011}
S.~Awodey, R.~Garner,
  P.~Martin-L{\"o}f, and V.~Voevodsky, eds.,
{\em The Homotopy interpretation of constructive type theory},
\newblock  Mathematisches Forschungsinstitut Oberwolfach Report No.
  11/2011.  Available from 
     \url{http://hottheory.files.wordpress.com/2011/06/report-11_2011.pdf}.

     \bibitem{Awodey-Warren:0709.0248}
S.~Awodey and M.~Warren, \emph{Homotopy theoretic models of identity
  types}, Math. Proc. Cambridge Philos. Soc. \textbf{146} (2009), 
  45--55. \arxiv{0709.0248}

  \bibitem{Barr-Wells}
  M.~Barr and C.~Wells, 
  \emph{Toposes, Triples and Theories},
\newblock Grundlehren der 
              Mathematischen Wissenschaften \textbf{278},
	      Springer-Verlag, 1985.
   Corrected reprint in \emph{Reprints in Theory and Applications of 
              Categories}, \textbf{12} (2005), x+288~pp.

\bibitem{Carboni-Janelidze-Kelly-Pare}
A.~Carboni, G.~Janelidze, G.~M.~Kelly, and R.~Par{\'e}.
\emph{On localization and stabilization for factorization systems},
Appl. Categ. Struct., \textbf{5} (1997),
1--58.

   
\bibitem{Cassidy-Hebert-Kelly}
C.~Cassidy, M.~H{\'e}bert, and G.~M.~Kelly.
\emph{Reflective subcategories, localizations 
and factorization systems}, J. Austral. Math. Soc. Ser. A, \textbf{38} (1985), 
287--329.

\bibitem{Cisinski:THT}
D.-C.~Cisinski, \emph{Th\'eories homotopiques dans les topos}, J. Pure Appl. 
Algebra \textbf{174} (2002), 43--82.

\bibitem{Cisinski:1406.0058}
D.-C.~Cisinski, \emph{Univalent universes for elegant models of homotopy types},
\newblock Preprint, \arxiv{1406.0058}.

\bibitem{Cisinski:blogpost}
D.-C.~Cisinski and M.~Shulman, entry at the n-Category Caf\'e, 
\\
\url{http://golem.ph.utexas.edu/category/2012/05/the_mysterious_nature_of_right.html#c041306}.

\bibitem{DS}
D.~Dugger and D.~Spivak,
\emph{Mapping spaces in quasi-categories},
\newblock Algebr. Geom. Topol. {\bf 11},
(2011), 263--325. \arxiv{0911.0469}.

\bibitem{DK84}
W.~G.~Dwyer and D.~M.~Kan,
\emph{Homotopy theory and simplicial groupoids},
\newblock Nederl. Akad. Wetensch. Indag. Math. {\bf 46},
(1984), 379--385.

\bibitem{Gambino-Garner:0803.4349}
N.~Gambino and R.~Garner, 
\emph{The identity type weak factorisation system}, 
\newblock Theoret. Comput. Sci. \textbf{409} (2008), 94--109.
  \arxiv{0803.4349}.

\bibitem{Garner-Lack:1106.5331}
R.~Garner and S.~Lack, \emph{Grothendieck quasitoposes},
\newblock J. Algebra {\bf 355} (2012), 111--127.
\newblock \arxiv{1106.5331}.

\bibitem{Gepner-Haugseng:1312}
D.~ Gepner and R.~Haugseng,
\newblock  {\em Enriched $\infty$-categories via non-symmetric $\infty$-operads},
\newblock Adv. Math. {\bf 279} (2015), 575--716.
\newblock \arxiv{1312.3178}.


\bibitem{Hofmann-Streicher:98}
M.~Hofmann and T.~Streicher, \emph{The groupoid interpretation of type
  theory}, in: \emph{Twenty-five years of constructive type theory},  Oxford
  Logic Guides  \textbf{36}, Oxford University Press, 1998, pp.~83--111.

  \bibitem{Joyal:CRM}
A.~Joyal, {\em The theory of quasi-categories}, Advanced Course on 
Simplicial Methods in Higher Categories, vol. II, Quaderns, n\'um. 45 (2008),
pp.~147--497, CRM Barcelona,  available at
\url{http://mat.uab.cat/~kock/crm/hocat/advanced-course/Quadern45-2.pdf}.

\bibitem{Kapulkin-Lumsdaine:1211.2851}
K.~Kapulkin and P.~L.~Lumsdaine,
\newblock {\em The simplicial model of univalent foundations (after Voevodsky)}.
\newblock Preprint, \arxiv{1211.2851}.

\bibitem{Kapulkin-Lumsdaine-Voevodsky:1203.2553}
K.~Kapulkin, P.~L.~Lumsdaine, and V.~Voevodsky,
\newblock {\em Univalence in simplicial sets}.
\newblock Preprint, \arxiv{1203.2553}.

\bibitem{Lumsdaine:0812.0409}
P.~L.~Lumsdaine, \emph{Weak {$\omega$}-categories from intensional type
  theory}, Log. Methods Comput. Sci. \textbf{6} 3:24 (2010), 1--19.
\arxiv{0812.0409}.

\bibitem{Lumsdaine-Warren:1411.1736}
P.~L.~Lumsdaine and M.~Warren.
\emph{The local universes model: an overlooked coherence construction for
dependent type theories}.
\newblock  ACM Trans. Comput. Log. {\bf 16} (2015), Art. 23, 31pp.
\newblock \arxiv{1411.1736}.


\bibitem{MacLane-Moerdijk}
  S.~Mac Lane and I.~Moerdijk, 
  \emph{Sheaves in Geometry and Logic: a first introduction to topos theory},
  \newblock Springer-Verlag, 1995.

\bibitem{Morel-Voevodsky:IHES}
F.~Morel and V.~Voevodsky, \emph{{${\bf A}^1$}-homotopy theory of
  schemes}, Inst. Hautes \'Etudes Sci. Publ. Math. {\bf 90} (1999), 45--143
  (2001).

  \bibitem{nLab:}
nLab entry,
Model of type theory in an (infinity,1)-topos.
\url{http://ncatlab.org/homotopytypetheory/show/model+of+type+theory+in+an+%28infinity%2C1%29-topos}


  
  \bibitem{Rezk:MR1804411}
C.~Rezk.
\newblock {\em A model for the homotopy theory of homotopy theory}.
\newblock Trans. Amer. Math. Soc. {\bf 353} (2001), 973--1007.


\bibitem{Shulman:1203.3253}
M.~Shulman,
\newblock {\em The univalence axiom for inverse diagrams and homotopy canonicity}.
\newblock Math. Struct. Comput. Sci. {\bf 25} (2015), 1203--1277.
\newblock \arxiv{1203.3253}.

\bibitem{Shulman:1307.6248}
M.~Shulman,
\newblock {\em The univalence axiom for elegant Reedy presheaves}.
\newblock Homology, Homotopy Appl. {\bf 17} (2015), 81--106.
\arxiv{1307.6248}.

\bibitem{SO}
M.~Spitzweck and P.~A.~{\O}stv{\ae}r,
\newblock {\em Motivic twisted K-theory}.
\newblock Algebr. Geom. Topol. {\bf 12} (2012), 565--599.
\arxiv{1008.4915}.

\bibitem{Streicher:MR3166196}
T.~Streicher,
\newblock \emph{A model of type theory in simplicial sets:
a brief introduction to {V}oevodsky's homotopy type theory}.
\newblock J. Appl. Log. {\bf 12} (2014), 45--49.

\bibitem{HoTT-book}
The Univalent~Foundations Program.
\newblock {\em Homotopy type theory---univalent foundations of mathematics}.
\newblock Institute for Advanced Study, Princeton, NJ, 2013.
\newblock Available from \url{http://homotopytypetheory.org/book}.

\bibitem{vdBerg-Garner:0812.0298}
B.~van~den Berg and R.~Garner,
\newblock \emph{Types are weak $\omega$-groupoids},
\newblock Proc.  London Math. Soc.
  {\bf 102} (2011), 370--394.
  \arxiv{0812.0298}

\bibitem{VoevodskyV:notts}
V.~Voevodsky,
\newblock {\em Notes on type systems}, 2011.
\newblock Available from
\url{http://www.math.ias.edu/~vladimir/Site3/Univalent_Foundations.html}.


\stepcounter{enumiv}\bibitem[HTT]{Lurie:HTT}
J.~Lurie, {\em Higher Topos Theory},  Annals of Mathematics Studies
  \textbf{170}, Princeton University Press, Princeton, NJ, 2009, xviii+925~pp.
\newblock Available from \url{http://www.math.harvard.edu/~lurie/}.

\stepcounter{enumiv}\bibitem[HA]{Lurie:HA}
J.~Lurie, {\em Higher Algebra}.
Available from \url{http://www.math.harvard.edu/~lurie/}.


\end{thebibliography}
\end{document}